\documentclass[a4paper,12pt]{article}
\usepackage{amsfonts}
\usepackage{amsmath}
\usepackage{amssymb}
\usepackage{graphicx}
\usepackage{epsf}
\usepackage{epsfig}

\usepackage{color}
\usepackage{afterpage}
\usepackage{verbatim}

\usepackage[small]{caption}
\newenvironment{proof}[1][Proof]{\textbf{#1.} }{\ \rule{0.5em}{0.5em}}
\newtheorem{theorem}{Theorem}[section]

\usepackage{times}

\begin{document}

\title{Numerical studies  of the Lagrangian approach for reconstruction of the conductivity in a waveguide}

\author{Larisa Beilina  \thanks{Department of Mathematical Sciences,
    Chalmers  University of Technology and Gothenburg University,
    SE-412 96   Gothenburg Sweden, e-mail \texttt{larisa.beilina@chalmers.se}} \and K. Niinim\"aki \thanks{
IR4M UMR8081, CNRS, University of Paris-Sud, University of
Paris-Saclay, SHFJ, 4 place du G\'{e}n\'{e}ral Leclerc 91401 Orsay
France, e-mail \texttt{kati.niinimaki@u-psud.fr}}}

\date{}

\maketitle 

\begin{abstract}
We consider an inverse problem of reconstructing the
  conductivity function in a hyperbolic equation using single
  space-time domain noisy observations of the solution on the
  backscattering boundary of the computational domain. We formulate
  our inverse problem as an optimization problem and use Lagrangian
  approach to minimize the corresponding Tikhonov functional.  We
  present a theorem of a local strong convexity of our functional and
  derive error estimates between computed and regularized as well as
  exact solutions of this functional, correspondingly.
In numerical simulations we apply domain decomposition finite
element-finite difference method for minimization of the
Lagrangian. Our computational study shows efficiency of the proposed
method in the reconstruction of the conductivity function in three
dimensions.

\end{abstract}
\section{Introduction}
\label{sec:intro}
In this work, we consider the coefficient inverse problem (CIP) 
 of reconstructing the conductivity function in a hyperbolic equation
using single observation of the solution of this equation in space
and time on the backscattered boundary of the computational domain. In
our simulations, backscattered boundary measurements are generated by a
single direction of propagation of a plane wave.  We solve our CIP via
minimization of the corresponding Tikhonov functional and use
the Lagrangian approach to minimize it.  Applying results of \cite{BOOK,
  BKK}, we have formulated a theorem of a local strong convexity of this
functional in our case and show that the gradient method for
minimizing this functional will converge.
  We have also presented estimates of the norms between
computed and regularized solution of the Tikhonov functional via the
$L_2$ norm of the Fr\'{e}chet derivative of this functional and via
the corresponding Lagrangian.  In the minimization procedure of the
Lagrangian, we applied conjugate gradient method and the domain
decomposition finite element/finite difference method of
\cite{BAbsorb}. The method of \cite{BAbsorb} is convenient for our
simulations since it is efficiently implemented in the software package
WavES \cite{waves} in C++ using PETSc \cite{petsc} and message
passing interface (MPI).

We tested our iterative algorithm by reconstructing a
conductivity function that represents some small scatterers as well
as smooth function inside the domain of interest.
 In all of our numerical simulations of this work we
induced one non-zero initial condition in the hyperbolic equation
accordingly to the theory of the recent work \cite{CristofolLiSoc}. In
\cite{CristofolLiSoc} it was shown 
 that one non-zero initial condition associated with the observation of 
the solution of the hyperbolic equation involve uniqueness and stability results in
reconstruction of the conductivity function for a cylindrical domains.
Our three-dimensional numerical simulations show that we can accurately reconstruct
large contrast of the conductivity function  as well as its location. In our future
work, similar to \cite{BJ, BTKB}, we are planning to use
an adaptive finite element method in order to improve reconstruction
of the shapes obtained in this work.

Another method for reconstruction of conductivity function - a
layer-stripping algorithm with respect to pseudo-frequency - 
was presented in \cite{Chow}.
 In addition the mathematical model governed by the hyperbolic
 equation studied in this work can also be considered as a special
 case of a time-dependent transverse magnetic polarized wave
 scattering problem or as a simplified acoustic wave model for fluids
 with variable density and a constant bulk modulus.  In recent
 years, some rapid identification techniques have been developed for
 solving the elastodynamic inverse problem, for instance, crack/fault
 identification techniques are developed for cracks having free
 boundary condition using a reciprocity gap function \cite{Bui2005,
   Bui2004}, and linear sampling techniques are designed to locate
 inclusions in the isotropic elastic medium \cite{Bellis, Fata}. 
 To compare performance of the
 algorithm of this paper with different algorithms of \cite{Bellis, Bui2005,
   Bui2004, Chow, Fata} can be the subject of a future work.

The paper is organized as follows. In section \ref{sec:model} we
formulate the forward and inverse problems.  In section
\ref{sec:tikhonov} we present the Tikhonov functional to be minimized
and formulate the theorem of a local strong convexity of this
functional.  Section 4 is devoted to a Lagrangian approach to solve
the inverse problem.  In section \ref{sec:fem} we present finite
element method for the solution of our optimization problem and
formulate conjugate gradient algorithm used in computations. Finally,
in section \ref{sec:Numer-Simul} we present results of reconstructing
the conductivity function in three dimensions.

\section{Statement of the forward and inverse problems}

\label{sec:model}

Let $\Omega \subset \mathbb{R}^{3}$ be a convex bounded domain with
the boundary $\partial \Omega \in C^{3}$, $x=(x_1, x_2, x_3) \in
\mathbb{R}^{3}$ and $C^{k+\alpha }$ is H\"older space, $k\geq 0$ is an
integer and $\alpha \in \left( 0,\,1\right).$ We use the notation $\Omega_T
:= \Omega \times (0,T), \partial \Omega_T := \partial \Omega \times
(0,T), T > 0$.  Next, in our theoretical and numerical investigations
we use domain decomposition of the domain $\Omega$ into two
subregions, $\Omega_{\rm IN}$ and $\Omega_{\rm OUT}$ such that $\Omega
= \Omega_{\rm IN} \cup \Omega_{\rm OUT}$,  $\partial \Omega_{\rm IN} \subset
\partial \Omega_{\rm OUT}$, see figure \ref{fig:fig1}.  The
communication between two domains is done through  two layers of
structured nodes as described in \cite{BAbsorb}.
 The boundary $\partial \Omega$ of the domain $\Omega$ is  such that $\partial \Omega
=\partial _{1} \Omega \cup \partial _{2} \Omega \cup \partial _{3} \Omega$  where
$\partial _{1} \Omega$ and $\partial _{2} \Omega$ are, respectively, front and
back sides of the domain $\Omega$. The boundary $\partial _{3} \Omega$ is the union
of the left, right, top and bottom sides of the domain $\Omega$.

We denote by
 $S_T := \partial_1 \Omega \times (0,T)$ the space-time boundary
where we will have time-dependent observations of the backscattered field.
We use the notation $S_{1,1} := \partial_1
 \Omega \times (0,t_1]$, $S_{1,2} := \partial_1 \Omega \times
 (t_1,T)$,  $S_2 := \partial_2 \Omega \times (0, T)$,  $S_3 :=
 \partial_3 \Omega \times (0, T)$.

Our model problem is   as follows

\begin{equation}\label{model1}
\begin{split}
\frac{\partial^2 u}{\partial t^2}  -  \nabla \cdot ( c \nabla  u)   &= 0,~ \mbox{in}~~ \Omega_T, \\
  u(x,0) = f_0(x), ~~~ \partial_t u(x,0) &= f_1(x)~ \mbox{in}~~ \Omega,     \\
\partial _{n} u& = p\left( t\right) ~\mbox{on}~ S_{1,1},
\\
\partial _{n}  u& =-\partial _{t} u ~\mbox{on}~ S_{1,2},
\\
\partial _{n} u& =-\partial _{t} u~\mbox{on}~ S_2, \\
\partial _{n} u& =0~\mbox{on}~ S_3,\\
\end{split}
\end{equation}
which satisfies stability and uniqueness results of  \cite{CristofolLiSoc}.

Here, $p\left( t\right)\not\equiv 0$ is the incident plane wave
generated at the plane $\left\{x_3=x_{0}\right\}$ and propagating along
the $x_3$-axis. We assume that
\begin{equation}
f_{0}\in H^{1}(\Omega), f_{1}\in L_{2}(\Omega). \label{f1} 
\end{equation}

We assume that in $\Omega_{\rm OUT}$ the function $c(x)$ is known and
is defined as a constant coefficient $c = 1$.  For numerical solution of
the problem (\ref{model1}) in $\Omega_{\rm OUT}$ we can use either the
finite difference or the finite element method. Further in our
theoretical considerations we will use the finite element method in
both $\Omega_{\rm OUT}$ and  $\Omega_{\rm IN}$,  with known function
$c = 1$ in $\Omega_{\rm OUT}$ and in the overlapping layer of the structured
nodes between $\Omega_{\rm OUT}$ and $\Omega_{\rm IN}$. This layer is
constructed in a similar  way as in \cite{BAbsorb}. We note that in the numerical
simulations of section \ref{sec:Numer-Simul} we use the domain
decomposition method of \cite{BAbsorb} since this method is
efficiently implemented in the software package WavES \cite{waves} and
is convenient for our purposes.  We also note that both finite element
and finite difference techniques provide the same explicit schemes in
$\Omega_{\rm OUT}$ in the case of structured mesh in $\Omega_{\rm
  OUT}$, see \cite{Brenner} for details.

We make the following assumptions on the 
coefficient $c \left( x\right)$ in the problem (\ref{model1}):
\begin{equation} \label{2.3}
\begin{split}
Y = \{ c \left( x\right) &\in \left[ 1,d \right],~~ d = const.>1,~ c(x) =1
\text{ for }x\in  \Omega\diagdown \Omega_{\rm IN}, \\
c \left( x\right) &\in C^{2}\left( \bar{\Omega} \right) \}. 
\end{split}
\end{equation}

We consider the following

\textbf{Inverse Problem  (IP)} \emph{Suppose that the coefficient
}$c \left( x\right)$  \emph{\ of (\ref{model1})  satisfies conditions (\ref{2.3}).
 }\emph{\ Assume that the
  function  }$ c\left( x\right) $\emph{\ is unknown in the
  domain }$\Omega \diagdown  \Omega_{\rm OUT}$\emph{. Determine the function }$ c\left(
x\right) $\emph{\ for }$x\in \Omega \diagdown  \Omega_{\rm OUT},$ \emph{\ assuming that the
  following space and time-dependent function }$\tilde u\left( x,t\right) $\emph{\ is known}
\begin{equation}
  u\left( x,t\right) = \tilde u \left( x,t\right) ,\forall \left( x,t\right) \in
  S_T.  \label{2.5}
\end{equation}

From the assumptions (\ref{2.3}) it follows that we should know a priori upper and lower bounds 
of the function  $c\left(
x\right)$.  This corresponds to the theory of inverse problems about the
availability of a priori information for an ill-posed problem
\cite{Engl, tikhonov}. 
In applications, the assumption $c\left(x\right) =1$ for $x\in
\Omega_{\rm OUT}$ means that the function $c\left( x\right)$ corresponds
to the homogeneous domain in $\Omega \diagdown \Omega_{\rm OUT}$.

\begin{figure}[tbp]
 \begin{center}
 \begin{tabular}{cc}
{\includegraphics[scale=0.2, clip = true, trim = 8.0cm 8.0cm 8.0cm 8.0cm]{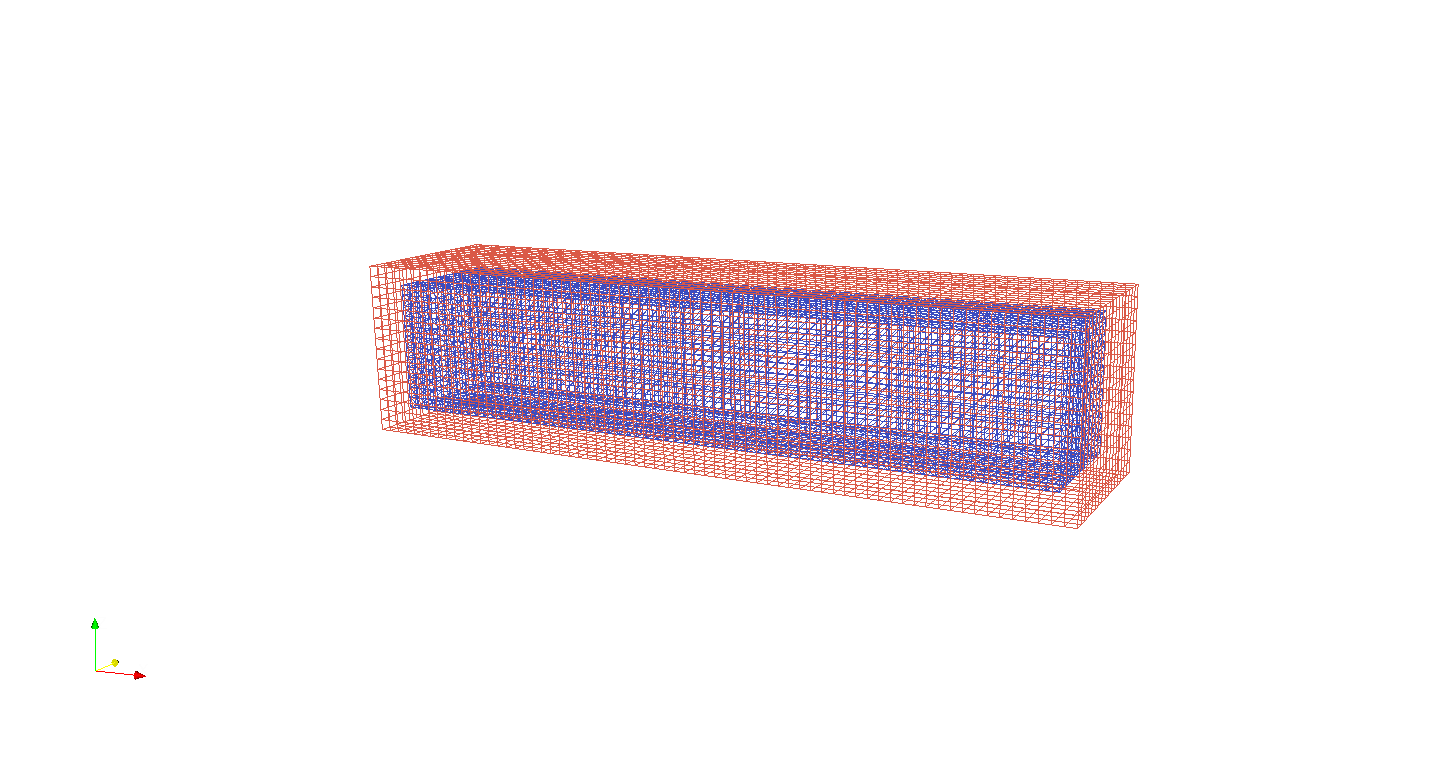}} &
 {\includegraphics[scale=0.2, clip = true, trim = 8.0cm 8.0cm 8.0cm 8.0cm]{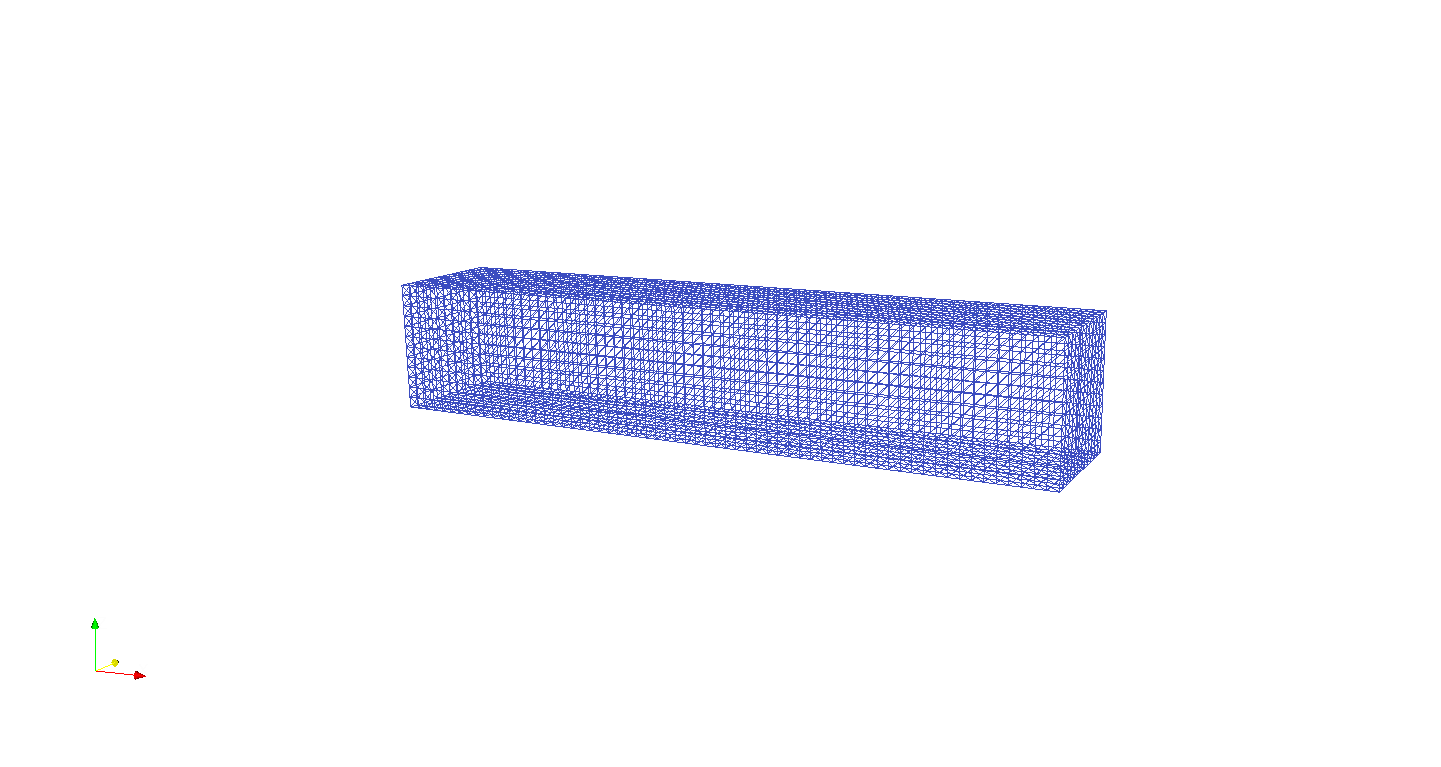}}
 \\
 a) $\Omega = \Omega_{IN} \cup \Omega_{OUT}$ &  b)  $\Omega_{IN}$  \\
 {\includegraphics[scale=0.2, clip = true, trim = 8.0cm 8.0cm 8.0cm 8.0cm]{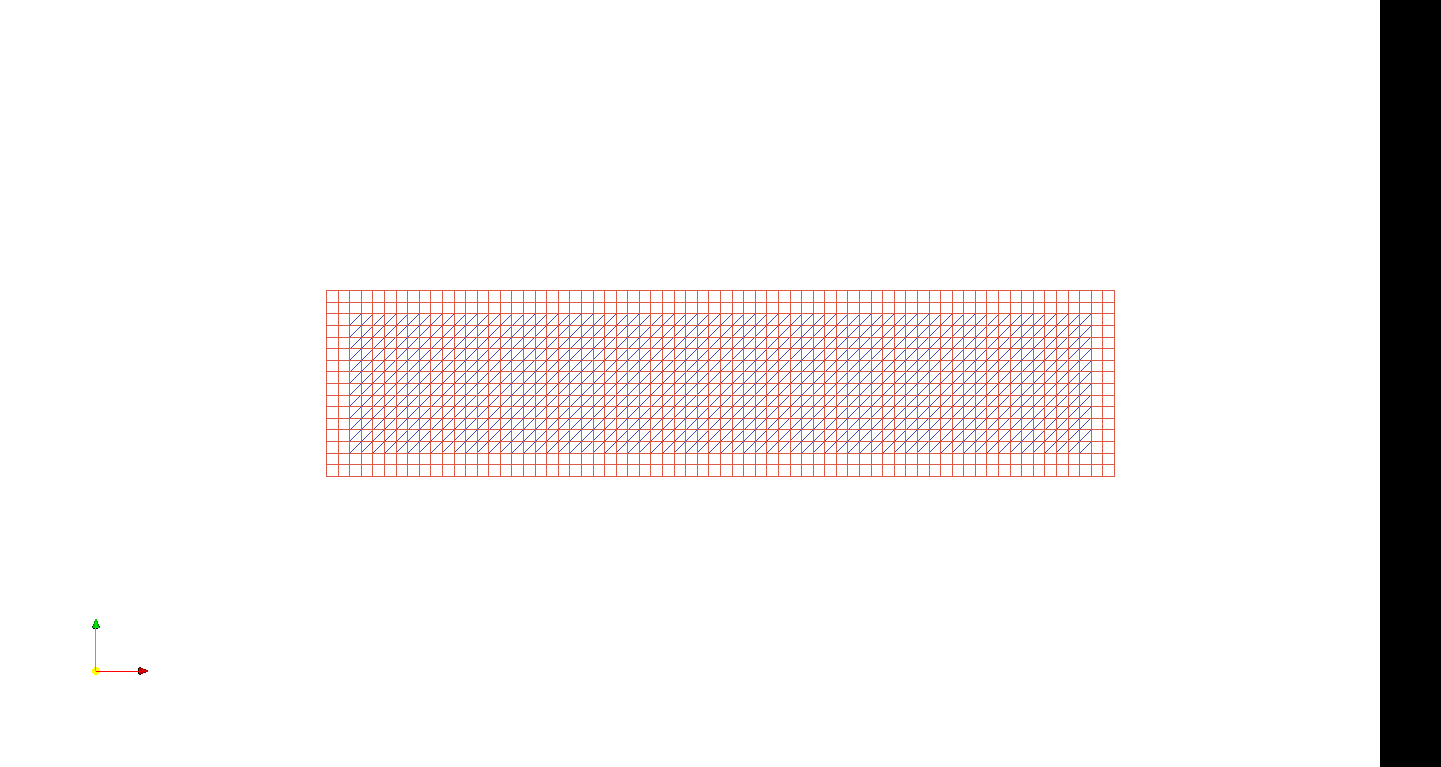}} &
   {\includegraphics[scale=0.2, clip = true, trim = 8.0cm 8.0cm 8.0cm 8.0cm]{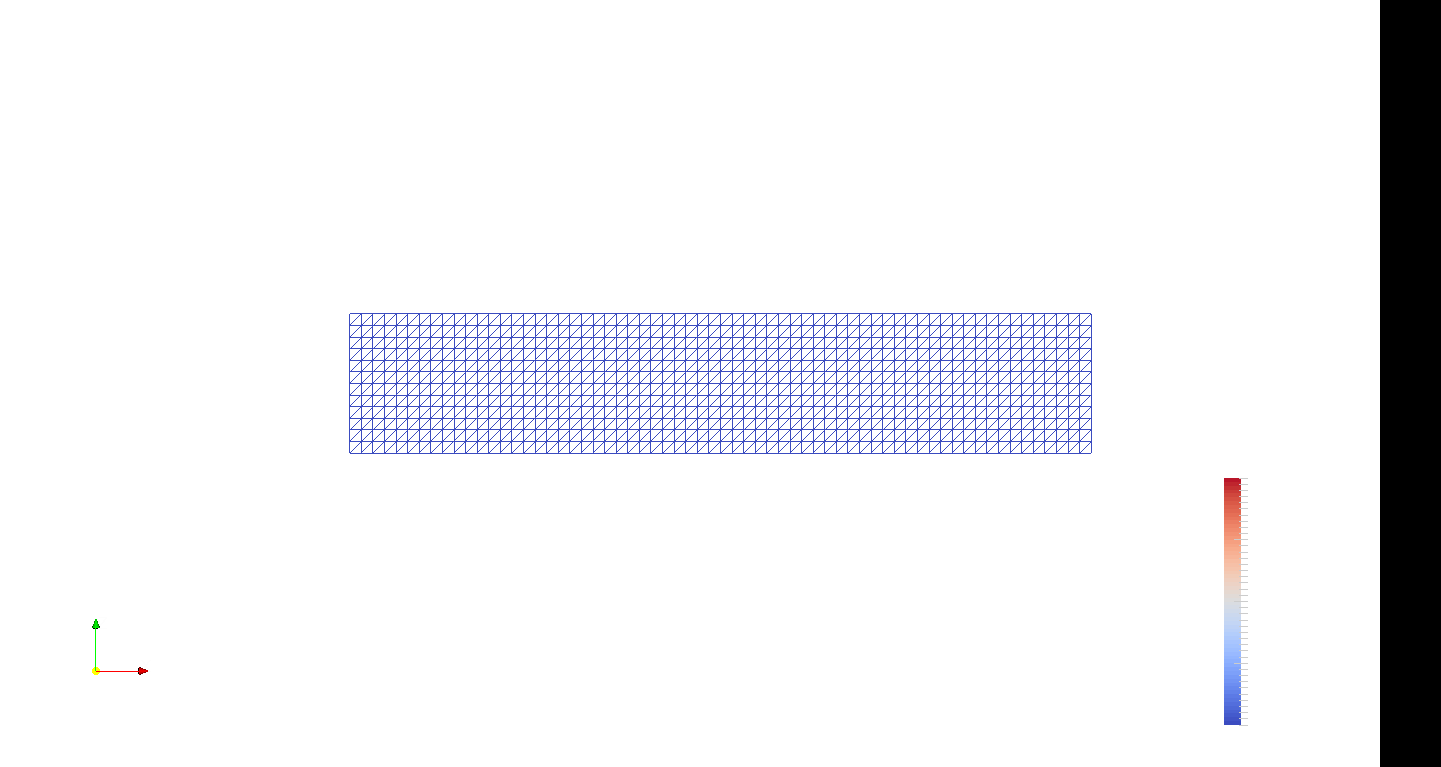}}
 \\
 c) $\Omega = \Omega_{IN} \cup \Omega_{OUT},~ x_1 x_2$ view &  d)  $\Omega_{IN},~ x_1 x_2$ view  \\
 
 \end{tabular}
 \end{center}
 \caption{{\protect\small \emph{a) The hybrid domain  $\Omega= \Omega_{IN} \cup \Omega_{OUT}$. Here, wireframe of $\Omega_{IN}$ is outlined in blue color and  wireframe of $\Omega_{OUT}$ in red color. b)  Wireframe of the inner domain $\Omega_{IN}$. }}}
 \label{fig:fig1}
 \end{figure}
 
\section{Tikhonov functional}

\label{sec:tikhonov}

We reformulate our inverse problem 
as an optimization problem and we seek for the function $c(x) \in
\Omega \diagdown \Omega_{\rm OUT} $. This function should fit to
the space-time observations $\tilde u$ measured at $\partial_1 \Omega$.
Thus, we  minimize  the Tikhonov functional
\begin{equation}
J(c) := J(u, c) = \frac{1}{2} \int_{S_T}(u - \tilde{u})^2 z_{\delta }(t) d \sigma dt +
\frac{1}{2} \gamma  \int_{\Omega}(c -  c_0)^2~~ dx,
\label{functional}
\end{equation}
where $\tilde{u}$ is the observed field in (\ref{2.5}), $u$ satisfies 
(\ref{model1}), $c_{0}$ is the initial guess for $c$, and $\gamma$ is
the regularization parameter.  Here, $z_{\delta }(t)$ is a cut-off
function to impose compatibility
conditions at $\overline{\Omega}_{T}\cap \left\{ t=T\right\} $ for the
adjoint problem (\ref{adjoint1}) which is defined as in
\cite{BCN}.

Let us define the $L_2$ inner product and the norm in $\Omega_T$ and
 $\Omega$, respectively, as
\begin{eqnarray*}
((u,v))_{\Omega_T}    &=& \int_{\Omega} \int_0^T u v~ dx dt, \\
||u||^2  &=& ((u,u))_{\Omega_T}, \\
(u,v)_{\Omega}    &=& \int_{\Omega} u v~ dx, \\
|u|^2  &=& (u,u)_{\Omega}.
\end{eqnarray*}
We also introduce the following spaces of real valued vector functions
\begin{equation}
\begin{split}
H_u^1 &:= \{ w \in H^1(\Omega_T):  w( \cdot , 0) = f_0(x), \partial_t w(\cdot, 0)=f_1(x) \}, \\
H_{\lambda}^1 &:= \{ w \in  H^1(\Omega_T):  w( \cdot , T) = \partial_t w(\cdot, T)= 0\},\\
U^{1} &=H_{u}^{1}(\Omega_T)\times H_{\lambda }^{1}(\Omega_T)\times C\left( \overline{\Omega}\right),\\
U^{0} &=L_{2}\left(\Omega_{T}\right) \times L_{2}\left(\Omega_{T}\right) \times
L_{2}\left( \Omega \right). 
\end{split}
\end{equation}
In our theoretical investigations below we need to reformulate the results of
\cite{BOOK, BKK} for the case of our \textbf{IP}.
Below in this section, $||\cdot||$  denotes $L_2$ norm.

We introduce a noise level $\delta $ in the function $\tilde{u}(x,t)$ in
the Tikhonov functional (\ref{functional}) that corresponds to the
theory of ill-posed problems \cite{BKS, BK,  tikhonov};
\begin{equation}
\tilde{u}(x,t)=  u(x,t, c^*) + \tilde{u}_{\delta
}(x,t);\text{ }  u(x,t, c^*), \tilde{u}_{\delta }\in L_{2}\left(
S_{T}\right), \label{4.247}
\end{equation}
where $u(x,t, c^*)$ is the exact data corresponding to the exact function $c^*$  in (\ref{model1}), and the function $\tilde{u}_{\delta }(x,t)$
represents the error in these data. In other words, we can write 
\begin{equation}
\left\Vert \tilde{u}_{\delta }\right\Vert _{L_{2}\left( S_{T}\right) }\leq \delta .
\label{4.248}
\end{equation}

Let  $Q_2 = L_{2}(S_{T})$ and  $Q_1$ be the finite dimensional
linear space such that
\begin{equation}
Q_1 = \bigcup_{K_h} \mathrm{span} (V(K_h)),
\end{equation}
and
\begin{equation}
V(K_h) = \{ v(x): v(x) \in H^1(\Omega) \},
\end{equation}
where $K_h$ is the finite-element mesh  defined in section \ref{sec:fem}.

Let $G \subset Q_1$ be a closed bounded convex set satisfying conditions \eqref{2.3}.
We introduce the operator $F:G \rightarrow L_{2}(S_T) $ corresponding to the Tikhonov functional (\ref{functional}) as 
\begin{equation}
F(c) \left( x,t\right) := u\mid _{S_T}\in
L_{2}\left(S_T \right),  \label{48}
\end{equation}
where $u(x,t,c) := u(x,t)$ is the weak solution of the problem (\ref{model1})
and thus, depends on the function  $c$.

We impose  assumption that the operator $F:G \rightarrow L_{2}(S_T) $ is one-to-one.
Next, we assume that there exists the exact solution $c^{\ast }\in
G$ of the equation
\begin{equation}
F\left( c^{\ast }\right) = u\left(
x,t,c^{\ast }\right) \mid _{S_T}.  \label{49}
\end{equation}
It follows from our assumption that the operator $F:G \rightarrow L_{2}(S_T)$ is one-to-one and thus, for a given function $ u\left(
x,t,c^{\ast }\right),$ this solution is unique.

We denote by 
\begin{equation}\label{neigh}
V_{d}\left(c\right) =\left\{c^{\prime }\in Q_1:\left\| c^{\prime } - c \right\|   < d,~~ \forall d >0 ~~\forall c  \in Q_1\right\} .
\end{equation}
 We also
assume  that the operator $F$ has the Lipschitz continuous Frech\'{e}t
derivative $F^{\prime }(c) $ for $c \in V_{1}(c^{\ast}),$  such that there exist constants  $N_{1},N_{2} > 0$
\begin{equation}
\left\| F^{\prime }(c) \right\| \leq N_{1},\left\| F^{\prime
}(c_1)  - F^{\prime }(c_2) \right\| \leq N_{2}\left\|
 c_1 - c_2 \right\| ,\forall c_1, c_2 \in V_{1}\left(c^{\ast }\right) .  \label{2.7}
\end{equation}

similar to \cite{BKK} we choose the constant $D= D\left(N_{1},N_{2}\right) =const.>0$ such that
\begin{equation}
\left\| J^{\prime }\left(c_1\right) -J^{\prime }\left(c_2\right) \right\| \leq D\left\| c_1 - c_2\right\| ,\forall c_1, c_2 \in V_{1}(c^*).
\label{2.10}
\end{equation}
Through the paper, similar to \cite{BKK},  we assume that
\begin{eqnarray}
\left\| c_0 - c^{\ast }\right\| &\leq &\delta ^{\xi},~\xi =const.\in \left( 0,1\right) ,  \label{2.11} \\
\gamma &=&\delta ^{\zeta}, \zeta=const.\in \left( 0,\min (\xi, 2( 1- \xi) ) \right ),  \label{2.12}
\end{eqnarray}
where  $\gamma$ is the regularization parameter in (\ref{functional}).
Equation (\ref{2.11}) means that we assume that  initial
guess $c_0$  in (\ref{functional}) is located in a sufficiently
small neighborhood $V_{\delta ^{\xi}}(c^*)$
of the exact solution $c^*$.
From Lemma 2.1 and 3.2 of \cite{BKK} it follows that
conditions (\ref{2.11})- (\ref{2.12}) ensures that  $(c^{\ast }, c_0)$ belong to
an appropriate neighborhood of the regularized solution of the functional (\ref{functional}).

Below we reformulate Theorem 1.9.1.2 of \cite{BOOK} for the Tikhonov
functional (\ref{functional}). Different proofs of this theorem can be found in
\cite{BOOK}, \cite{BakKB} and in \cite{BKK} and are straightly applied to our   case.

The question of stability and uniqueness of our \textbf{IP} is
addressed in \cite{CristofolLiSoc} for the case of the unbounded domain.

 \begin{theorem} \label{theorem1}

   Let $Q_1, Q_2$ are two Hilbert spaces such that $\dim Q_1 <
   \infty$, $G \subset Q_1$ is a closed bounded convex set  satisfying
   conditions \eqref{2.3}, $Q_2 = L_{2}(S_{T})$ and $F:G \rightarrow
   Q_2 $
   is a continuous one-to-one operator.

   Assume that
   the conditions (\ref{4.247})- (\ref{4.248}), (\ref{2.7})-(\ref{2.10}) hold.
   Assume that there exists the exact
   solution $c^* \in G$  of the equation $F(c^*) =0$  for the case of
   the   exact data $u(x,t,c^*)$  in (\ref{4.247}).  Let  the
   regularization parameter  $\gamma$  in (\ref{functional}) be   such that 
\begin{equation}\label{regpar}
\gamma = \gamma(\delta) =\delta ^{2\nu},~~\nu  =const.\in \left( 0,\frac{1}{4}\right), ~\quad \forall \delta
\in \left( 0,1 \right).
\end{equation}
Let  $c_0$  satisfies the  condition  (\ref{2.11}). Then the Tikhonov functional  (\ref{functional})    is strongly convex in the neighborhood $V_{\gamma} \left( \delta \right) (c^*)$  with the strong convexity constant  $\alpha =  \gamma = \delta ^{2\nu}$  such that
\begin{equation}
\left\Vert c_{1} - c_{2}\right\Vert ^{2}\leq \frac{2}{\delta ^{2\nu }}\left(
J'(c_1) - J'(c_2), c_{1} - c_{2}\right),
\text{ }\forall c_1, c_2 \in Q_1.  \label{4.249}
\end{equation}
Next, there exists   the unique regularized solution  $c_{\gamma}$  of
the   functional (\ref{functional}) and this solution  $c_{\gamma} \in
V_{\delta ^{3\nu}/3}(c^*).$ The   gradient method of the
minimization of the functional   (\ref{functional}) which starts at
$c_0$  converges to the regularized solution  $c_{\gamma}$ of this   functional
 and
\begin{equation}\label{accur}
\left\Vert c_{\gamma}  - c^* \right\Vert \leq \xi
\left\Vert c_0 - c^* \right\Vert, ~~\xi \in (0,1).
\end{equation}
    
 \end{theorem}

The property (\ref{accur}) means that the regularized solution of the Tikhonov
functional (\ref{functional}) provides a better accuracy than the
initial guess $c_0$  if it satisfies condition
(\ref{2.11}).

The next theorem presents the estimate of the norm
$\left\Vert c - c_{\gamma} \right\Vert $
via the norm of the Fr\'{e}chet derivative of the Tikhonov functional (\ref{functional}).

 \begin{theorem} \label{theorem2}
Assume that the conditions of Theorem \ref{theorem1} hold. Then for any function $c
\in V_{\gamma(\delta)}(c^*)$  the following error estimate is valid
\begin{equation}
\begin{split}
\left\Vert c - c_{\gamma(\delta)} \right\Vert 
&\leq \frac{2}{\delta ^{2\nu }}
\left\Vert P_h J^{\prime }(c) \right\Vert \leq \frac{2}{\delta ^{2 \nu }} \left\Vert J^{\prime }(c)\right\Vert\\
&
= \frac{2}{\delta ^{2\nu }} \left\Vert L^{\prime }_c(v(c))\right\Vert =\frac{2}{\delta ^{2\nu }} \left\Vert \int_0^T (\nabla  u(c)) (\nabla  \lambda(c))~dt 
+\gamma  (c - c_0)\right\Vert ,
  \label{4.250}
\end{split}
\end{equation}
where  $c_{\gamma(\delta)}$   is the  minimizer of the  Tikhonov functional (\ref{functional})  computed with the regularization parameter   $\gamma$ and $
P_h: L_{2}\left( \Omega \right) \rightarrow Q_1$  is the operator of
orthogonal projection  of the space $L_{2}\left( \Omega \right) $ on its subspace $Q_1$,  $L^{\prime }_c(u(c))$ is the  Fr\'{e}chet derivative of the Lagrangian (\ref{lagrangian1})  given by (\ref{derfunc}).
 \end{theorem}

 \begin{proof}

Since $ c_{\gamma} :=c_{\gamma(\delta)}$ is the minimizer of the
functional (\ref{functional}) on the set $G$ and $c_{\gamma} \in
Int\left( G\right),$ then $P_hJ^{\prime }(c_{\gamma}) = 0$, or using
(\ref{derfunc}) we can write
\begin{equation}\label{4.2511}
\begin{split}
P_hJ^{\prime }_{c}(c_{\gamma}) &= 0.
\end{split}
\end{equation}

similar to Theorem 4.11.2 of \cite{BOOK}, since $ c  -  c_{\gamma}  \in Q_1,$ then
\begin{equation*}
\begin{split}
(J^{\prime }(c) &- J^{\prime } (c_{\gamma}), c - c_{\gamma})  =
(P_h J^{\prime }(c) - P_h J^{\prime}(c_{\gamma}),  c - c_{\gamma}).
\end{split}
\end{equation*}
Hence, using  (\ref{4.2511}) and the strong convexity property (\ref{4.249})  we can write that
\begin{equation*}
\begin{split}
\left\Vert c -  c_{\gamma}  \right \Vert ^{2} 
&\leq \frac{2}{\delta ^{2\nu}}
\left( J^{\prime }(c) - J^{\prime }(c_{\gamma}), c - c_{\gamma}\right) \\
&=\frac{2}{\delta^{2\nu }}\left( P_h J^{\prime }(c)  - P_h J^{\prime }(c_{\gamma}), c - c_{\gamma}\right) \\
&=\frac{2}{\delta ^{2\nu }}( P_h J^{\prime }(c),  c - c_{\gamma}) \\ 
&\leq
\frac{2}{\delta ^{2\nu }}\left\Vert P_h J ^{\prime }(c)
 \right\Vert  \cdot \left\Vert c - c_{\gamma} \right\Vert.
\end{split}
\end{equation*}
Thus,  from the expression above we get
\begin{equation}\label{expres}
\left\Vert  c - c_{\gamma} \right\Vert ^{2} \leq \frac{2}{
\delta ^{2\nu }} \left \Vert
 P_h J^{\prime }(c) \right\Vert \cdot \left\Vert  c - c_{\gamma} \right\Vert.
\end{equation}

Using the fact 
\begin{equation*}
\left\Vert P_h J^{\prime}(c) \right\Vert_{L_{2}\left( \Omega \right)} \leq 
\left\Vert J^{\prime }(c ) \right\Vert _{L_{2}\left( \Omega \right)}
\end{equation*}
  together with (\ref{derfunc})  and (\ref{derfunc2}) and dividing the
  expression (\ref{expres})  by $\left\Vert c - c_{\gamma}
  \right\Vert $, we obtain the  inequality (\ref{4.250}). 
 \end{proof}

In our final theorem we  present the error between the computed and exact solutions of the functional  (\ref{functional}).

 \begin{theorem} \label{Theorem3}
    Assume that the conditions of Theorem \ref{theorem1}   hold. Then for any
    function $c \in V_{\gamma(\delta)}(c^*)$  the following error estimate holds
\begin{equation}\label{theorem3}
\left\Vert c - c^* \right\Vert 
\leq \frac{2}{\delta ^{2\nu }} \left\Vert \int_0^T (\nabla  u(c)) (\nabla  \lambda(c))~dt 
+\gamma  (c - c_0)\right\Vert  + \xi \left\Vert c_0 - c^* \right\Vert. 
\end{equation}
 \end{theorem}

 \begin{proof}
Applying Theorem  \ref{theorem2} and inequality (\ref{accur})
 we get  the inequality (\ref{theorem3})
\begin{equation}
\begin{split}
\left\Vert c - c^* \right\Vert &= \left\Vert c - c_{\gamma(\delta)} +  c_{\gamma(\delta)} - c^* \right\Vert \leq  \left\Vert c - c_{\gamma(\delta)}\right\Vert  + \left\Vert c_{\gamma(\delta)} - c^* \right\Vert  \\
&\leq \frac{2}{\delta ^{2\nu }} \left\Vert \int_0^T (\nabla  u(c)) (\nabla  \lambda(c))~dt 
+\gamma  (c - c_0)\right\Vert  + \xi \left\Vert c_0 - c^* \right\Vert.
\end{split}
\end{equation}

 \end{proof}

\section{Lagrangian approach}

In this section, we will present the Lagrangian approach to solve the inverse
problem \textbf{IP}.  
To minimize the Tikhonov functional (\ref{functional}) we introduce the Lagrangian
\begin{equation}\label{lagrangian1}
\begin{split}
L(v) &= J(u, c) 
-  \int_{\Omega_T} \frac{\partial
 \lambda }{\partial t} \frac{\partial u}{\partial t}  ~dxdt  
+   \int_{\Omega_T}( c \nabla  u)( \nabla  \lambda)~dxdt  \\
&  - \int_{\Omega} \lambda(x,0) f_1(x) ~dx - \int_{S_{1,1}} \lambda p(t) ~d \sigma dt    + \int_{S_{1,2}} \lambda \partial_t u ~d\sigma dt
   + \int_{S_2} \lambda \partial_t u  ~d\sigma dt   , \\
\end{split}
\end{equation}
where $v=(u,\lambda, c) \in U^1$.
 We search for a stationary point of (\ref{lagrangian1}) 
with respect to $v$ satisfying $ \forall \bar{v}= (\bar{u}, \bar{\lambda}, \bar{c}) \in U^1$
\begin{equation}
 L'(v; \bar{v}) = 0 ,  \label{scalar_lagr1}
\end{equation}
where $ L^\prime (v;\cdot )$ is the Jacobian of $L$ at $v$.
We can rewrite the equation (\ref{scalar_lagr1}) as
\begin{equation}
 L'(v; \bar{v}) = \frac{\partial L}{\partial \lambda}(v)(\bar{\lambda}) + \frac{\partial L}{\partial u}(v)(\bar{u})  +  \frac{\partial L}{\partial c}(v)(\bar{c}) = 0.  \label{scalar_lagr}
\end{equation}

 To find the Frech\'{e}t derivative (\ref{scalar_lagr1}) of the
 Lagrangian (\ref{lagrangian1}) we consider $L(v + \bar{v}) - L(v)~
 \forall \bar{v} \in U^1$. Then we single out the linear part of the
 obtained expression with respect to $ \bar{v}$.  When we derive the
 Frech\'{e}t derivative we assume that in the Lagrangian
 (\ref{lagrangian1}) function $v=(u,\lambda, c) \in U^1$ can be
 varied independently on each other.
  We assume that $\lambda \left(
 x,T\right) =\partial _{t}\lambda \left( x,T\right) =0$ and seek to
 impose  conditions on the function $\lambda $ such that $ L\left(
 u,\lambda ,c \right) :=L\left( v\right) = J\left( u,c\right).$ Next,
 we use the fact that $\lambda (x ,T) = \frac{\partial
   \lambda}{\partial t} (x,T) =0$ and $u(x,0)= f_0(x), \frac{\partial
   u}{\partial t} (x ,0) = f_1(x)$, as well as $c=1$ on $\partial
 \Omega$, together with boundary conditions $ \partial_n u = 0$ and $
 \partial_n \lambda = 0$ on $S_3$.  The equation (\ref{scalar_lagr1})
 expresses that for all $\bar{u}$,
\begin{equation}\label{forward1}
\begin{split}
0 = \frac{\partial L}{\partial \lambda}(u)(\bar{\lambda}) =
&- \int_{\Omega_T}  \frac{\partial \bar{\lambda}}{\partial t} \frac{\partial u}{\partial t}~ dxdt 
+  \int_{\Omega_T}  ( c \nabla u) (\nabla  \bar{\lambda}) ~ dxdt 
- \int_{\Omega} \bar{\lambda}(x,0) f_1(x) ~dx  \\
&- \int_{S_{1,1}} \bar{\lambda} p(t) ~d \sigma dt   + \int_{S_{1,2}}
\bar{\lambda} \partial_t u ~d\sigma dt    + \int_{S_2} \bar{\lambda}
\partial_t u  ~d\sigma dt,\\
&\forall \bar{\lambda} \in H_{\lambda}^1(\Omega_T),\\
\end{split}
\end{equation}
\begin{equation} \label{control1}
\begin{split}
0 = \frac{\partial L}{\partial u}(u)(\bar{u}) &=
\int_{S_T}(u - \tilde{u})~ \bar{u}~ z_{\delta}~ d \sigma dt
-\int_{\Omega}  \frac{\partial \lambda}{\partial t} (x,0) \bar{u}(x,0) dx  
 - \int_{S_{1,2} \cup S_2} \frac{\partial \lambda}{\partial t} \bar{u} ~d\sigma dt \\   
&-  \int_{\Omega_T}  \frac{\partial \lambda}{\partial t} \frac{\partial \bar{u}}{\partial t}~ dxdt
 + \int_{\Omega_T} ( c \nabla  \lambda) (\nabla  \bar{u})  ~ dxdt,~~\forall \bar{u} \in H_{u}^1(\Omega_T).\\
\end{split}
\end{equation}
Finally, we obtain the equation which expresses 
stationarity  of the gradient
with respect to  $c$  :
\begin{equation} \label{grad1} 
0 = \frac{\partial L}{\partial c}(u)(\bar{c})
 =    \int_{\Omega_T} (\nabla  u) (\nabla  \lambda) \bar{c} ~dxdt 
+\gamma \int_{\Omega} (c - c_0) \bar{c}~dx,~ x \in \Omega.
\end{equation}

The equation (\ref{forward1}) is the weak formulation of the state equation
(\ref{model1}) and the equation (\ref{control1}) is the weak
formulation of the following adjoint problem
\begin{equation}
\begin{split} \label{adjoint1}
 \frac{\partial^2 \lambda}{\partial t^2} - 
  \nabla \cdot (c \nabla  \lambda)  &= -  (u - \tilde{u})|_{S_T} z_{\delta} ~  \mbox{ in } \Omega_T,   \\
\lambda(\cdot, T)& =  \frac{\partial \lambda}{\partial t}(\cdot, T) = 0, \\
\partial _{n} \lambda& = \partial _{t} \lambda,~\mbox{on}~ S_{1,2},
\\
\partial _{n} \lambda& = \partial _{t} \lambda,~\mbox{on}~ S_2, \\
\partial _{n} \lambda& =0,~\mbox{on}~ S_3.
\end{split}
\end{equation}
We note that we have positive sign here in absorbing boundary
conditions. However, after
discretization in time of these conditions we will obtain the same
schemes for computation of $\lambda^{k-1}$ as for the computation of
$u^{k+1}$ in the forward problem since we solve the adjoint problem
backward in time.

Let now the functions  $u(c), \lambda(c)$ be the exact solutions of the
forward and adjoint problems, respectively, for the known function $c$
satisfying condition (\ref{accur}). Then with $v(c) = (u(c),
\lambda(c), c) \in U^1$  and using the fact that  for exact solutions  $u(c), \lambda(c)$ from (\ref{lagrangian1}) we have
\begin{equation}
J( u(c), c) = L(v(c))
\end{equation}
~ and assuming that solutions $u(c), \lambda(c) $  are sufficiently stable (see Chapter 5 of book \cite{lad} for details),   we can write that the  Frech\'{e}t derivative of the Tikhonov functional is  given by
\begin{equation}\label{derfunc}
\begin{split}
J'(c) := J'(u(c), c) &=  \frac{\partial J}{\partial c}(u(c), c)
  =  \frac{\partial L}{\partial c}(v(c))
.
\end{split}
\end{equation}
Inserting (\ref{grad1})  into  (\ref{derfunc})
we get
\begin{equation} \label{derfunc2}
\begin{split}
J'(c)(x) &:= J'(u(c),c)(x) =
 \int_0^T (\nabla  u(c)) (\nabla  \lambda(c))  (x,t)~dt 
+\gamma  (c - c_0)(x).
\end{split}
\end{equation}

We note that the Lagrangian (\ref{lagrangian1}) and the optimality
conditions (\ref{forward1}), (\ref{control1}) will be the same, when
the homogeneous initial conditions are used in the model problem
(\ref{model1}), and only the terms containing the initial conditions will
disappear.

\section{Finite element method for the solution of an optimization problem}
\label{sec:fem}

In this section, we formulate the finite element method for the
solution of the forward problem (\ref{model1}) and the adjoint problem
(\ref{adjoint1}). We also  present a conjugate gradient method for the
solution of our \textbf{IP}.

\subsection{Finite element discretization}
\label{sec:fem1}

 We discretize $\Omega_{FEM} \times (0,T)$ denoting by $K_h = \{K\}$
 the partition of
 the domain $\Omega_{FEM}$ into tetrahedra $K$ ($h=h(x)$ being a mesh function,
 defined as $h |_K = h_K$, representing the local diameter of the elements),
 and we let $J_{\tau}$ be a partition of the time interval $(0,T)$ into time
sub-intervals $J=(t_{k-1},t_k]$ of uniform length $\tau = t_k - t_{k-1}$. We
 assume also a minimal angle condition on the $K_h$ \cite{Brenner}.

To formulate the finite element method,  we
 define the finite element spaces $C_h$, $W_h^u$ and $W_h^{\lambda}$.
First we introduce the finite element trial space $W_h^u$ for  $u$ defined by
\begin{equation}
W_h^u := \{ w \in H_u^1: w|_{K \times J} \in  P_1(K) \times P_1(J),  \forall K \in K_h,  \forall J \in J_{\tau} \}, \nonumber
\end{equation}
where $P_1(K)$ and $P_1(J)$ denote the set of piecewise-linear functions on $K$
and $J$, respectively.
We also introduce the finite element test space  $W_h^{\lambda}$ defined by
\begin{equation}
W_h^{\lambda} := \{ w \in H_{\lambda}^1: w|_{K \times J} \in  P_1(K) \times P_1(J),  \forall K \in K_h,  \forall J \in J_{\tau} \}. \nonumber
\end{equation}

To approximate function
$c(x)$   we will use the space of piecewise constant functions $C_{h} \subset L_2(\Omega)$, 
\begin{equation}\label{p0}
C_{h}:=\{u\in L_{2}(\Omega ):u|_{K}\in P_{0}(K),\forall K\in  K_h\}, 
\end{equation}
where $P_{0}(K)$ is the piecewise constant function on $K$.

Next, we define $V_h = W_h^u \times W_h^{\lambda} \times C_h$.
Usually $\dim V_{h}<\infty $ and $V_{h}\subset U^{1}$ as a set and we
consider $V_{h}$ as a discrete analogue of the space $U^{1}.$ We
introduce the same norm in $V_{h}$ as the one in $U^{0},\left\Vert
\bullet \right\Vert _{V_{h}}:=\left\Vert \bullet \right\Vert
_{U^{0}}$, from which it follows that in finite dimensional spaces all
norms are equivalent and in our computations we compute coefficients
in the space $C_h$.  The finite element method now reads: Find $v_h
\in V_h$, such that
\begin{equation}
L'(v_h)(\bar{v})=0 ~\forall
\bar{v} \in V_h .  \label{varlagr}
\end{equation}

Using (\ref{varlagr}) we can write the finite element method for the
forward problem (\ref{model1}) (for convenience we will use here and in section \ref{sec:discrete}  $f_0=f_1=0$ in $\Omega_T$): Find $u_h
\in W_h^u$, such that $\forall \bar{\lambda} \in W_h^\lambda$ and for known $c_h \in C_h$,
\begin{equation}\label{varforward}
\begin{split}
&- \int_{\Omega_{T}}  \frac{\partial \bar{\lambda}}{\partial t} \frac{\partial u_h}{\partial t}~ dxdt 
- \int_{S_{1,1}} p(t) \bar{\lambda} ~ d\sigma dt \\
&+  \int_{S_{1,2} \cup S_2} \partial_t u_h  \bar{\lambda} ~ d\sigma dt
+  \int_{\Omega_{T}}  ( c_h \nabla u_h) (\nabla  \bar{\lambda}) ~ dxdt = 0.
\end{split}
\end{equation}
Similarly, the finite element method for the  adjoint problem (\ref{adjoint1}) in $\Omega_T$ reads: Find $\lambda_h \in W_h^\lambda$, such that $\forall \bar{u} \in W_h^u$ and for known $u_h \in W_h^u$, $c_h \in C_h$,
\begin{equation} \label{varadjoint}
\begin{split}
&-  \int_{\Omega_{T}}  \frac{\partial \lambda_h}{\partial t} \frac{\partial \bar{u}}{\partial t}~ dxdt +  \int_{S_T} (u_h - \tilde{u}) z_{\sigma} \bar{\lambda} 
~ d\sigma dt \\
&- \int_{S_{1,2} \cup S_2} \partial_t \lambda_h  \bar{u} ~ d\sigma dt
 + \int_{\Omega_{T}} ( c_h \nabla  \lambda_h) (\nabla  \bar{u})  ~ dxdt = 0.
\end{split}
\end{equation}

\subsection{Fully discrete scheme}
\label{sec:discrete}

We expand functions $u_h(x,t)$ and $\lambda_h(x,t)$ in terms of the
standard continuous piecewise linear functions
$\{\varphi_i(x)\}_{i=1}^M$ in space and $\{\psi_k(t)\}_{k=1}^N$ in
time, substitute them into (\ref{varforward}) and (\ref{varadjoint}),
and compute explicitly all time integrals which will appear in the
system of discrete equations. Finally, we obtain the following system
of linear equations for the forward and adjoint problems
(\ref{model1}), (\ref{adjoint1}), correspondingly (for
convenience we consider here $f_0=f_1=0$):
\begin{equation} \label{femod1}
\begin{split}
 M (\mathbf{u}^{k+1} - 2 \mathbf{u}^k  + \mathbf{u}^{k-1})  &=  
\tau^2  G^k    - \tau^2  K \mathbf{u}^k - \frac{1}{2}\tau M_{\partial \Omega} (\mathbf{u}^{k+1}- \mathbf{u}^{k-1}) ,   \\
M (\boldsymbol{\lambda}^{k+1} - 2 \boldsymbol{\lambda}^k + \boldsymbol{
\lambda}^{k-1}) &=  -\tau^2  S^k - \tau^2  K \boldsymbol{\lambda}^k + \frac{1}{2}\tau M_{\partial \Omega} (\boldsymbol{\lambda}^{k+1}- \boldsymbol{\lambda}^{k-1}), \\
\end{split}
\end{equation}
with initial  conditions :
\begin{eqnarray}
u(\cdot, 0)&= \frac{\partial u}{\partial t}(\cdot, 0) = 0, \\
\lambda (\cdot,T) &= \frac{\partial \lambda}{\partial t} (\cdot,T) =0.
\end{eqnarray}
  Here, $M$ and $M_{\partial \Omega}$ are the block mass matrix in space and mass matrix at the boundary $\partial \Omega$, respectively, $K$ is the block
  stiffness matrix, $G^k$ and $S^k$ are  load vectors at time level $t_k$,
  $\mathbf{u}^k$ and $ \boldsymbol{\lambda}^k$ denote the nodal values
  of $u_h(\cdot,t_k)$ and $\lambda_h(\cdot,t_k)$, respectively and $\tau$ is
  a time step.  For details of obtaining this system of discrete
  equations and computing the time integrals 
in it, as well as for obtaining then the system (\ref{femod1}), we
  refer to  \cite{BAbsorb}.

Let us define the mapping $F_K$ for the reference element $\hat{K}$
such that $F_K(\hat{K})=K$ and let $\hat{\varphi}$ be the piecewise
linear local basis function on the reference element $\hat{K}$ such
that $\varphi \circ F_K = \hat{\varphi}$.  Then the explicit formulas
for the entries in system (\ref{femod1}) at each element $K$ can be
given as:
\begin{equation}
\begin{split}
  M_{i,j}^{K} & =    ( ~\varphi_i \circ F_K, \varphi_j \circ F_K)_K, \\
  K_{i,j}^{K} & =   ( c_i \nabla  \varphi_i \circ F_K, \nabla  \varphi_j \circ F_K)_K,\\
 G_{j}^{K}&= (p^k, \varphi_j \circ F_K )_{K \in S_{1,1}}, \\
  S_{j}^{K}&= ((u_{h_{i,k}} - \tilde{u}_{i,k})|_{\partial_1 \Omega} z_{\delta}, \varphi_j \circ F_K )_{K}, \\
\end{split}
 \end{equation}
where $(\cdot,\cdot)_K$ denotes the $L_2(K)$ scalar product and
  $\partial K$ is the part of the boundary of element $K$ which lies at $\partial \Omega_{FEM}$.
Here, $u_{h_{i,k}}=u(x_i, t_k)$ are computed solutions of the forward
problem (\ref{model1}), and $\tilde{u}_{i,k}=\tilde{u}(x_i, t_k) $ are discrete
measured values of $\tilde{u}(x,t)$ at $S_T$ at the point $x_i \in K_h$ and time moment $t_k \in J_k$.

To obtain an explicit scheme we approximate $M$ with the lumped mass matrix
$M^{L}$ (for further details, see \cite{Cohen}). Next, we multiply (\ref{femod1}) by
$(M^{L})^{-1}$ and get the following  explicit method inside $\Omega_{FEM}$:
\begin{equation}   \label{fem}
\begin{split}
  \mathbf{u}^{k+1} = & - \tau^2 (M^{L})^{-1}  G^k + (2- \tau^2  (M^{L})^{-1} K)\mathbf{u}^k    -\mathbf{u}^{k-1},  \\
  \boldsymbol{\lambda}^{k-1} = &-\tau^2 (M^{L})^{-1} S^k  
  + (2  - \tau^2  (M^{L})^{-1} K) \boldsymbol{\lambda}^k  -\boldsymbol{\lambda}^{k+1}.  \\
\end{split}
\end{equation} 
In the formulas above the terms with $M_{\partial \Omega}$ disappeared since
we  used schemes  (\ref{fem}) only inside $\Omega_{FEM}$.

Finally, for reconstructing $c(x)$ we can use a gradient-based method
with an appropriate initial guess values of $c_0$ which satisfies 
the condition (\ref{2.11}).
  We have the following expression for the discrete
version of the gradient with respect to coefficient $c$ in
(\ref{grad1}):
\begin{equation} \label{gradient1}
g_h(x) =   \int_{0}^T \nabla u_h \nabla  \lambda_h dt + \gamma (c_h - c_0).
\end{equation}

Here, $\lambda_h$  and $u_h$ are computed values of the adjoint
and forward problems, respectively,  using explicit schemes
(\ref{fem}), and $c_h$   is approximated value of the computed
coefficient.

\subsection{The  algorithm}
\label{subsec:ad_alg}

We use conjugate gradient method for the iterative update of
approximations $c_{h}^{m}$ of the function $c_{ h}$, where $m$ is the
number of iteration in our optimization procedure. We denote
\begin{equation}
\begin{split}
{g}^m(x) = \int_{0}^T \nabla  u_h^m \nabla  \lambda_h^m  dt + \gamma (c_h^m - c_0), \\
\end{split}
\end{equation}
where functions $u_{h}\left( x,t,c_{h}^{m}\right) ,~\lambda _{h}\left(x,t,c_{h}^{m}\right) $\ are computed by solving
the state and the adjoint problems
with $c:=c_{h}^{m}$.\\

\textbf{Algorithm}

\begin{itemize}
\item[Step 0.] \hspace{0.6cm} Choose a mesh $K_{h}$ in $\Omega$ and a time partition $J$ of the time interval $\left( 0,T\right) .$
Start with the initial approximation $c_{h}^{0}= c_0$ and compute the
sequences of $c_{h}^{m}$ via the following steps:

\item[Step 1.]  \hspace{0.6cm} Compute solutions $u_{h}\left( x,t,c_{h}^{m}\right) $
  and $\lambda _{h}\left( x,t,c_{h}^{m}\right) $ of the state (\ref{model1})
   and the adjoint (\ref{adjoint1}) problems on $K_{h}$ and
  $J$ using explicit schemes (\ref{fem}).

\item[Step 2.]  \hspace{0.6cm} Update the coefficient $c_h:=c_{h}^{m+1}$ 
  on $K_{h}$ and $J$ using the conjugate gradient method
\begin{equation*}
\begin{split}
c_h^{m+1} &=  c_h^{m}  + \alpha^m d^m(x),\\
\end{split}
\end{equation*}
where $\alpha$ is the step-size in the gradient update \cite{Peron}
which is computed as
$$
 \alpha^m = \frac{((g^m, d^m))}{\gamma \| d^m \|^2},
$$
and
\begin{equation*}
\begin{split}
 d^m(x)&=  -g^m(x)  + \beta^m  d^{m-1}(x),
\end{split}
\end{equation*}
with
\begin{equation*}
\begin{split}
 \beta^m &= \frac{|| g^m(x)||^2}{|| g^{m-1}(x)||^2},
\end{split}
\end{equation*}
where $d^0(x)= -g^0(x)$.
\item[Step 3.]   \hspace{0.6cm} Stop computing $c_{h}^{m}$ and obtain the function
  $c_h$ if either $||g^{m}||_{L_{2}( \Omega)}\leq \theta$ or norms
  $||g^{m}||_{L_{2}(\Omega)}$ are stabilized. Here, $\theta$ is the
  tolerance in updates $m$ of the gradient method.  Otherwise set $m:=m+1$
  and go to step 1.

\end{itemize}

\section{Numerical Studies}
\label{sec:Numer-Simul}

In this section, we present numerical simulations of the reconstruction
of unknown function $c(x)$ of the equation (\ref{model1}) inside a
domain $\Omega_{FEM}$ using the algorithm of section
\ref{subsec:ad_alg}.

 For computations of the numerical approximations $u_h$ of the forward
 and $\lambda_h$ of the adjoint problems in step 1 of the algorithm of
 section \ref{subsec:ad_alg}, we use the domain decomposition method of
 \cite{BAbsorb}. We decompose $\Omega$ into two subregions
 $\Omega_{IN}$ and $\Omega_{OUT}$ as described in section
 \ref{sec:model}, and we define $\Omega_{FEM} := \Omega_{IN}$ and
 $\Omega_{FDM}:=\Omega_{OUT} $ such that $\Omega = \Omega_{FEM} \cup
 \Omega_{FDM}$.  In
 $\Omega_{FEM}$ we use finite elements as described in section
 \ref{sec:fem1}.
 In
 $\Omega_{FDM}$ we will use finite difference method.  The boundary
 $\partial \Omega$ is such that $\partial \Omega =\partial _{1} \Omega
 \cup \partial _{2} \Omega \cup \partial _{3} \Omega$, see section
 \ref{sec:model} for description of this boundary.

We assume that the conductivity function $c(x)$ is known inside
$\Omega_{FDM}$ and we set it to be $c(x) =1$.  The goal of our
numerical tests is to reconstruct small inclusions  with
$c=4.0$ inside every small scatterer, which can represent defects
inside a waveguide. We also test our reconstruction algorithm when $c(x)$ represents a
smooth function. We consider four different case studies with different
geometries of the scatterers:
\begin{itemize}
\item [i)] 3 scatterers of different size located on the same plane with respect to the wave propagation; 
\item [ii)] 3 scatterers of different size non-uniformly located inside the waveguide;
\item [iii)]~ $c(x)$ is smooth function  which is presented by one spike of Gaussian  function;
\item [iv)]~ $c(x)$ is smooth function presented by three  spikes of Gaussian  functions.
\end{itemize}
 Figures \ref{fig:fig2} and
\ref{fig:exact_gaus} present the considered geometries of the case studies.

In \cite{BAbsorb} it  was shown that the best reconstruction results for our set-ups 
are  obtained  for the wave length $\lambda$ with the frequency  $\omega = 40$ in the initialization of a plane wave  in (\ref{f}). Thus, 
for all test cases i)-iv)
we  choose $\omega = 40$  in (\ref{f})  and solve
the model problem (\ref{model1}) with non-homogeneous initial
condition $f_0(x)$ and with $f_1(x) =0$ in (\ref{model1}).
In all our tests we  initialized initial conditions at backscattered side $\partial_1 \Omega$ as
\begin{equation}\label{initcond}
\begin{split}
u(x,0) &= f_0(x)=\exp^{-(x_1^2 + x_2^2 + x_3^3)}  \cdot \cos  t|_{t=0} = \exp^{-(x_1^2 + x_2^2 + x_3^3)}  , \\
\frac{ \partial u}{\partial t} (x,0) &=  f_1(x)= -\exp^{-(x_1^2 + x_2^2 + x_3^3)} \cdot  \sin t|_{t=0} \equiv 0.
\end{split}
\end{equation}

 The domain decomposition is done in the same way, as described above, for all
 of the case studies. Next, we introduce dimensionless spatial
 variables $x^{\prime}= x/\left(1m\right)$  such  that the domain
 $\Omega_{FEM}$ is transformed into dimensionless computational domain
 \begin{equation*}
 \Omega_{FEM} = \left\{ x= (x_1,x_2,x_3);~ x_1 \in (
 -3.2,3.2),~ x_2 \in (-0.6,0.6),~ x_3 \in (-0.6,0.6) \right\} .
 \end{equation*}
   The dimensionless size of our computational domain
 $\Omega$ for the forward problem is
 \begin{equation*}
 \Omega = \left\{ x= (x_1,x_2,x_3); ~x_1 \in (
 -3.4,3.4),~ x_2 \in (-0.8,0.8),~ x_3 \in (-0.8,0.8) \right\} .
 \end{equation*}
 The space mesh in $\Omega_{FEM}$ and in $\Omega_{FDM}$ consists of
 tetrahedral and cubes, respectively. 
  We choose the mesh size $h=0.1$
 in our geometries in the domain decomposition FEM/FDM method, as well as in the
 overlapping regions between $\Omega_{FEM}$  and $\Omega_{FDM}$.

  We generate backscattered measurements $\tilde{u}$ at $S_T$ in $\Omega$ by a
  single  plane wave $p(t)$ initialized at $\partial_1 \Omega$
  in time $T=[0,3.0]$ such that
 \begin{equation}\label{f}
 \begin{split}
 p\left( t\right) =\left\{ 
 \begin{array}{ll}
 \sin \left( \omega t \right) ,\qquad &\text{ if }t\in \left( 0,\frac{2\pi }{\omega }
 \right) , \\ 
 0,&\text{ if } t>\frac{2\pi }{\omega }.
 \end{array}
 \right. 
 \end{split}
 \end{equation}

For the generation of the simulated backscattered data  for cases i)-ii) we first define exact
function ${c(x)}=4$ inside small scatterers, and
$c(x)=1$ at all other points of the computational domain
$\Omega_{FEM}$.  The function $c(x)$ for cases iii)-iv) is defined in
sections \ref{sec:caseiii} and \ref{sec:caseiv}, respectively.  Next, we
solve the forward problem (\ref{model1}) on a locally refined mesh in
$\Omega_{FEM}$ in time $T=[0,3.0]$ with a plane wave as in
(\ref{f}). This allows us to avoid problem with variational crimes.
Since we apply explicit schemes (\ref{femod1})  in our computations, we use 
the time step $\tau=0.006$ which satisfies  the CFL condition, see details in
\cite{BAbsorb, CFL67}.

 \begin{figure}
 \begin{center}
 \begin{tabular}{cc}
 {\includegraphics[scale=0.22,angle=-90, trim = 0.0cm 0.0cm 0.0cm 0.0cm, clip=true,]{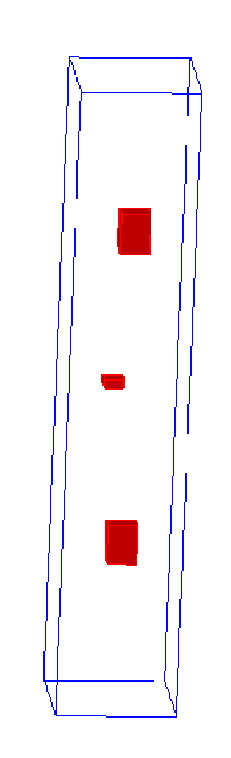}} &
 {\includegraphics[scale=0.22,angle=-90,  trim = 0.0cm 0.0cm 0.0cm 0.0cm, clip=true,]{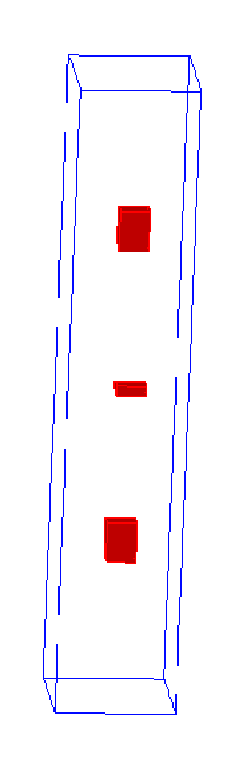}} 
 \\
 a) Test i)  & b) Test ii) \\
\end{tabular}
 \end{center}
 \caption{Exact
 values of the conductivity function in cases i) and  ii) are:   $c=4.0$ inside all small scatterers of a)-b), and
 $c=1.0$  everywhere else in $\Omega_{FEM}$.  }
 \label{fig:fig2}
 \end{figure}

 For all case studies, we start the optimization algorithm with guess
  values of the parameter $c(x)=1.0$ at all points in $\Omega$.  Such
  choice of the initial guess provides a good reconstruction for
  functions $c(x)$ and corresponds to starting the gradient algorithm
  from the homogeneous domain, see also \cite{BAbsorb, BCN,  BJ} for a similar
  choice of initial guess. In tests i)-ii) the minimal and maximal
  values of the functions $c(x)$ in our computations belongs to the
  following set of admissible parameters
 \begin{equation}\label{admpar}
 \begin{split}
  M_{c} \in \{c\in C(\overline{\Omega })|1\leq c(x)\leq 5\}.\\
 \end{split}
 \end{equation}

We regularize the solution of the inverse problem by starting
computations with regularization parameter $\gamma =0.01$ in
(\ref{functional}) which satisfies the condition (\ref{2.12}).
Our computational studies have shown that such choice of
the regularization parameter is optimal one for the solution of our IP
since it gives smallest relative error in the reconstruction of the
function $c(x)$. We refer to \cite{BKS, BK, Engl}, and references
therein, for different techniques for the choice of a regularization
parameter.  The tolerance $\theta$ at step 3 of our algorithm of
section \ref{subsec:ad_alg} is set to $\theta=10^{-6}$.
 
In our numerical simulations we have considered an additive noise
$\sigma$ introduced to the simulated boundary data $\tilde{u}$ in
(\ref{2.5}) as
\begin{equation}\label{additive}
u_{\sigma}\left( x^{i},t^{j}\right)  = \tilde{u}\left( x^{i},t^{j}\right) \left[ 1+ 
 \frac{\sigma }{100}\right].
\end{equation}
Here,
$x^{i}\in\partial \Omega $ is a mesh point at the boundary $\partial
\Omega ,\, t^{j}\in \left( 0,T\right) $ is a mesh point in the time mesh $J_{\tau}$,
and $\sigma$ is the
noise level in percents.

We use a post-processing procedure to get images of figures
 \ref{fig:rec_caseii},
\ref{fig:rec_caseiii}, \ref{fig:1gausnoise10} -
\ref{fig:3gausnoise10}. This procedure is as follows: assume, that
the functions $c^m(x)$ are our reconstructions obtained by the algorithm of
section \ref{subsec:ad_alg} where $m$ is the number of iterations in
the conjugate gradient algorithm when we have stopped to compute $c(x)$. Then to get
our final images, we set
 \begin{equation}\label{postproc}
 \widetilde{c}^m(x)=\left\{ 
 \begin{array}{ll}
 c^m(x), & \text{ if }c^m(x)>P \max\limits_{\Omega_{FEM} }c^m(x), \\ 
 1, & \text{ otherwise. }%
 \end{array}
 \right. 
 \end{equation}
The values of the parameter $P \in (0,1)$ depends on the concrete
reconstruction of the function $c(x)$ and plays the roll of a cut-off
parameter for the function $c(x)$.  If we choose $P \approx 1$ then we
will cut almost all reconstruction of the function $c(x)$. Thus,
values of $P$ should be chosen numerically.
For tests i), ii) we have used $P=0.7$ and for case studies
iii)-iv) we choose $P=0.5$.

\begin{table}[h] 
{\footnotesize Table 1. \emph{Computational results of the reconstructions
    in cases i)-iv) together with computational errors in achieved
    contrast in percents. Here, $\overline{N}$ is the final iteration
    number $m$ in the conjugate gradient method of section
    \ref{subsec:ad_alg}.}}  \par
\vspace{2mm}
\centerline{
\begin{tabular}{|c|c|}
 \hline
   $\sigma=3\%$ &  $\sigma = 10\%$ 
 \\
 \hline
\begin{tabular}{l|l|l|l} \hline
Case & $\max_{\Omega_{FEM}} c_{\overline{N}}$ &  error, \% & $\overline{N}$  \\ \hline
i) & $2.21 $ & 44.75 &$7  $   \\
ii) & $2.07$ &  48.25 & $7$   \\
iii) & $5.91$ & 1.5 & $12$   \\
iv) & $5.09$ & 15.2  &$15$   \\
\end{tabular}
 & 
\begin{tabular}{l|l|l|l} \hline
Case & $\max_{\Omega_{FEM}} c_{\overline{N}}$ &    error, \% & $\overline{N}$  \\ \hline
i) & $3.13$ & 21.75& $9 $   \\
ii) & $3.06$ & 23.5 & $9$   \\
iii) & $4.84$ & 19.3 & $16$   \\
iv) & $5.87$ & 2.2 &$18$   \\
\end{tabular} 
\\
\hline
\end{tabular}}
\end{table}

\subsection{Test case i)}

\label{sec:casei}

In this example we performed computations with two noise levels in
data: $\sigma = 3\%$ and $\sigma = 10 \%$. 
 Figure \ref{fig:data_caseii} presents typical
 behavior of noisy backscattered data in this case.
The results of
reconstruction for both noise levels are presented in figure
\ref{fig:rec_caseii}. We observe that the location of all inclusions
in $x_1 x_2$ direction is imaged very well. However, the location in the $x_3$
direction should still be improved.

 It follows from figure
\ref{fig:rec_caseii} and table 1 that the imaged contrast in the function  $c(x)$ is
$2.21:1=\max_{\Omega_{FEM}} c_{7}:1 $, where $n:=\overline{N}=7$ is our
final iteration number in the conjugate gradient method.  Similar observation is valid from figure
\ref{fig:rec_caseii} and table 1 for noise level 10 \% where imaged contrast in the function  $c(x)$ is
$3.13:1=\max_{\Omega_{FEM}} c_{9}:1 $, where $n:=\overline{N}=9$ is our
final iteration number.

 \begin{figure}[tbp]
 \begin{center}
 \begin{tabular}{ccc}
 {\includegraphics[scale=0.32,  clip=]{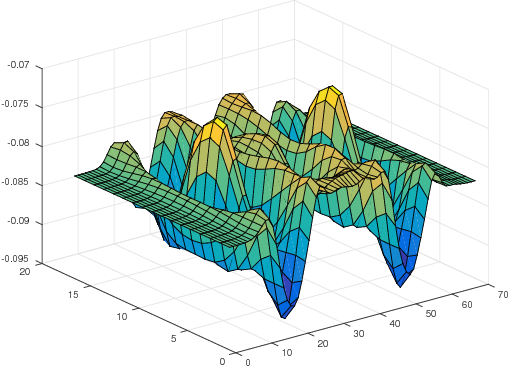}} &
 {\includegraphics[scale=0.32,  clip=]{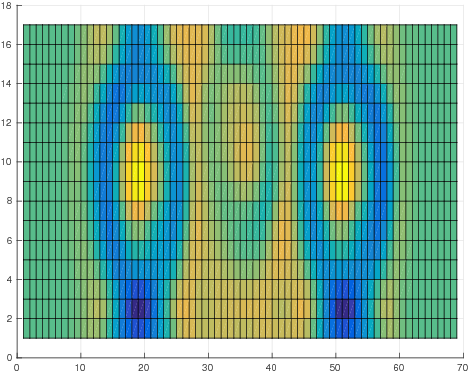}} \\
 a) prospect view & b)  $ x_1 x_2$  view \\
 \end{tabular}
 \end{center}
 \caption{Test case i). Behavior of the noisy backscattered data at time $t=1.8$ with $\sigma = 3 \%$ in (\ref{additive}). }
 \label{fig:data_caseii}
 \end{figure}

\begin{figure}
 \begin{center}
 \begin{tabular}{cc}
 {\includegraphics[scale=0.23, angle=-90, trim = 1.0cm 1.0cm 1.0cm 1.0cm, clip=true,]{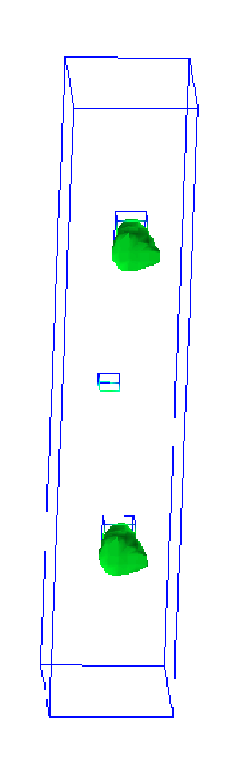}} &
 {\includegraphics[scale=0.23,angle=-90, trim = 1.0cm 1.0cm 1.0cm 1.0cm, clip=true,]{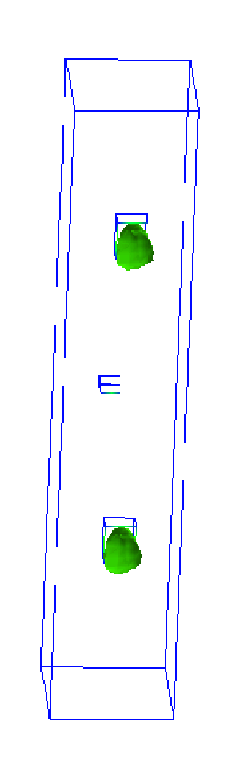}}  \\
$\max \limits_{\Omega_{FEM}} c(x) = 2.21, \sigma = 3 \% $ &  $\max \limits_{\Omega_{FEM}} c(x) = 3.13, \sigma = 10 \%$\\
 {\includegraphics[scale=0.23, angle=-90, trim = 1.0cm 1.0cm 1.0cm 1.0cm, clip=true,]{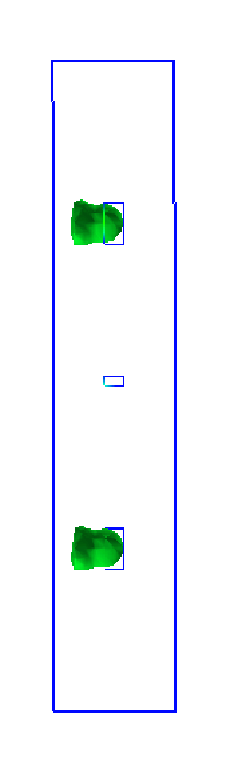}} &
 {\includegraphics[scale=0.23,angle=-90, trim = 1.0cm 1.0cm 1.0cm 1.0cm, clip=true,]{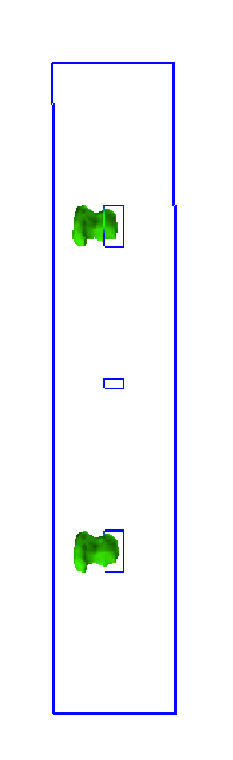}}  
\\
$x_1 x_3$ view & $x_1 x_3$ view\\
\end{tabular}
 \end{center}
 \caption{Test case i). Computed images of reconstructed $\tilde{c}$ for $\omega=40$ in
   (\ref{f}) and for a different noise level $\sigma$ in  (\ref{additive}). Bottom row present the respective  $x_{1}x_{3}$ views.}
 \label{fig:rec_caseii}
 \end{figure}

\subsection{Test case ii)}

\label{sec:caseii}

In this test, we have considered the same noise levels $\sigma$:
$\sigma = 3\%$ and $\sigma = 10 \%$, as in the test case i). The
behavior of the noisy backscattered data in this case is presented in
figure \ref{fig:data_caseiii}. Using figure \ref{fig:data_diff} we
observe that the difference in the amplitude of backscattered data
between the cases i) and ii) is very small and, as expected, is
located exactly at the place where the middle smallest inclusion of
figure \ref{fig:fig2} is moved,
This is because the difference in two
geometries of figure \ref{fig:fig2} is only in the location of the
small middle inclusion: in figure \ref{fig:fig2}-b) this inclusion is
moved more close to the backscattered boundary $\partial_1  \Omega$ than in the figure  \ref{fig:fig2}-a).

  The results of the reconstruction for both noise levels are presented in
  figure \ref{fig:rec_caseiii}. It follows from figure
\ref{fig:rec_caseiii} and table 1 that the imaged contrast in the function  $c(x)$ is
$2.07:1=\max_{\Omega_{FEM}} c_{7}:1 $, where $n:=\overline{N}=7$ is our
final iteration number in the conjugate gradient method when the noise level is 3 \%. Similar observation is valid from figure
\ref{fig:rec_caseiii} and table 1 for noise level 10 \% where imaged contrast in the function  $c(x)$ is
$3.06:1=\max_{\Omega_{FEM}} c_{9}:1 $, where $n:=\overline{N}=9$ is our
final iteration number. 
 Again, as in the case i) we observe
  that the location of all inclusions in $x_1 x_2$ direction is imaged
  very well. However, location in $x_3$ direction should still be
  improved.
 We also observe that the smallest inclusion of figure
 \ref{fig:fig2}-b) is reconstructed better than in the case i) since
 it is located closer to the observation boundary $\partial_1 S$.
Similar to \cite{BJ, BTKB}, in our future research, we plan to
 apply an adaptive finite element method which hopefully will improve
 the  shapes and sizes of all inclusions considered in tests i) and ii).

 \begin{figure}[tbp]
 \begin{center}
 \begin{tabular}{ccc}
 {\includegraphics[scale=0.32,  clip=]{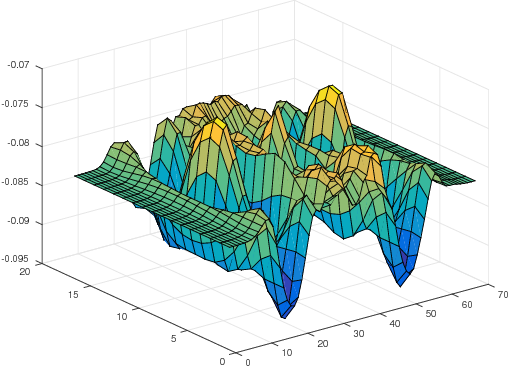}} &
 {\includegraphics[scale=0.32,  clip=]{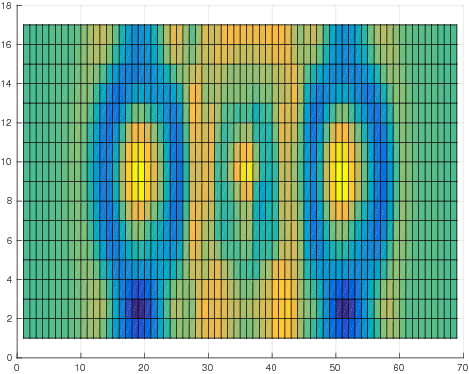}} \\
 a) prospect view & b) $x_1 x_2$ view \\
 \end{tabular}
 \end{center}
 \caption{Test case ii). Behavior of the noisy backscattered data at time $t=1.8$ with $\sigma = 3 \%$ in (\ref{additive}). }
 \label{fig:data_caseiii}
 \end{figure}

 \begin{figure}[tbp]
 \begin{center}
 \begin{tabular}{ccc}
 {\includegraphics[scale=0.32,  clip=]{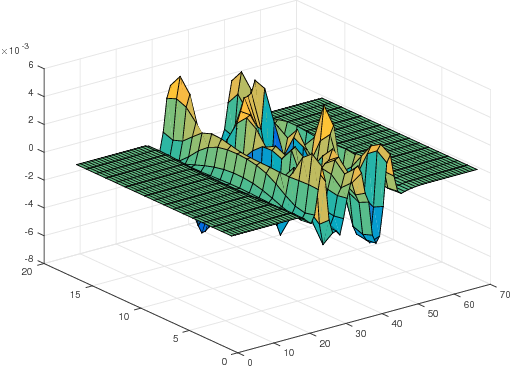}} &
 {\includegraphics[scale=0.32,  clip=]{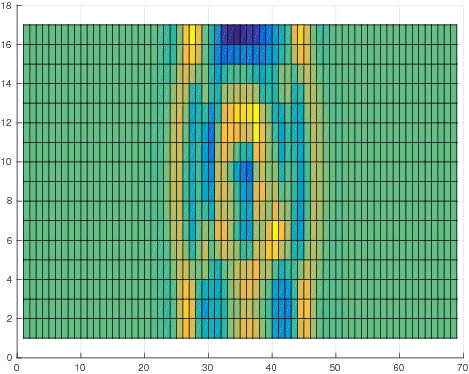}} \\
 a) prospect view & b) $x_1 x_2$ view \\
 \end{tabular}
 \end{center}
 \caption{The difference of the noisy backscattered data at time $t=1.8$ in case studies i) and ii) when noise level is $\sigma = 3 \%$ in (\ref{additive}). }
 \label{fig:data_diff}
 \end{figure}

 \begin{figure}
 \begin{center}
 \begin{tabular}{cc}
 {\includegraphics[scale=0.23, angle=-90, trim = 1.0cm 1.0cm 1.0cm 1.0cm, clip=true,]{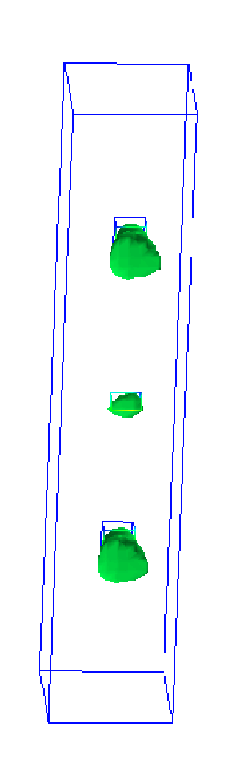}} &
 {\includegraphics[scale=0.23,angle=-90, trim = 1.0cm 1.0cm 1.0cm 1.0cm, clip=true,]{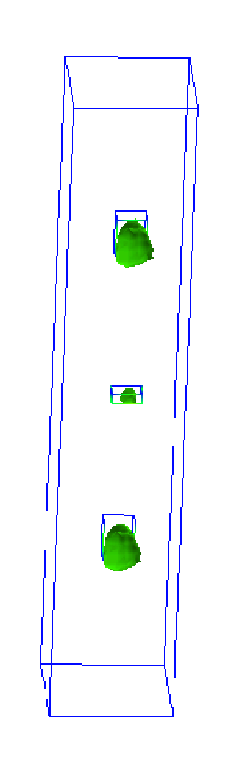}}  \\
$\max \limits_{\Omega_{FEM}} c(x) = 2.07, \sigma = 3 \% $ &  $\max \limits_{\Omega_{FEM}} c(x) = 3.06, \sigma = 10 \%$
\\
 {\includegraphics[scale=0.23, angle=-90, trim = 1.0cm 1.0cm 1.0cm 1.0cm, clip=true,]{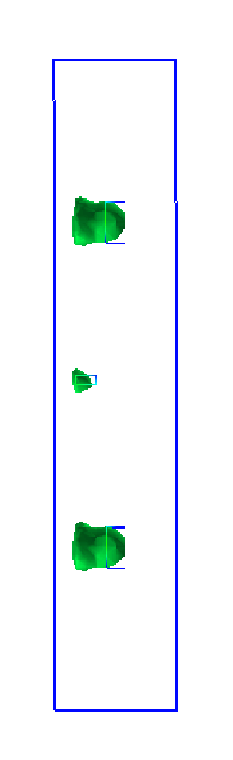}} &
 {\includegraphics[scale=0.23,angle=-90, trim = 1.0cm 1.0cm 1.0cm 1.0cm, clip=true,]{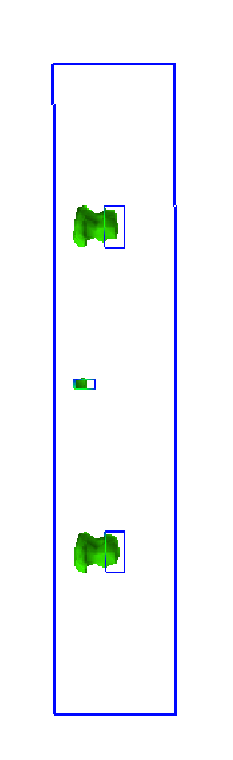}}  
\\
$x_1 x_3$ view &  $x_1 x_3$ view
\end{tabular}
 \end{center}
 \caption{Test case ii). Computed images of  reconstructed $\tilde{c}$ for $\omega=40$ in
   (\ref{f}) and for different noise level $\sigma$ in (\ref{additive}). Bottom row present the $x_{1}x_{3}$ views.}
 \label{fig:rec_caseiii}
 \end{figure}

 \begin{figure}
 \begin{center}
 \begin{tabular}{cc}
{\includegraphics[scale=0.19, trim = 0.0cm 3.0cm 0.0cm 3.0cm, clip=true,]{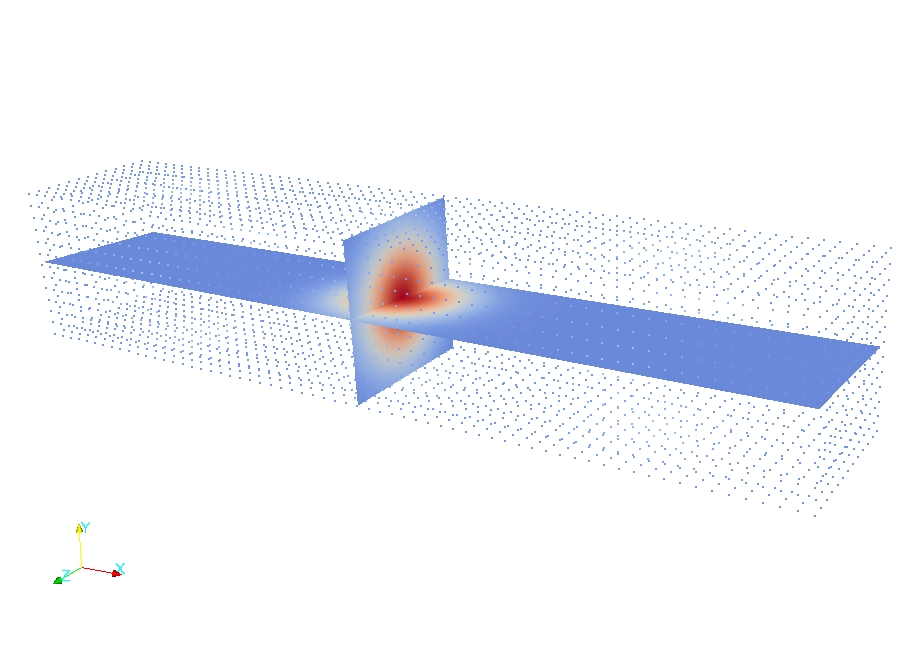}} &
  {\includegraphics[scale=0.19, trim = 0.0cm 3.0cm 0.0cm 3.0cm, clip=true,]{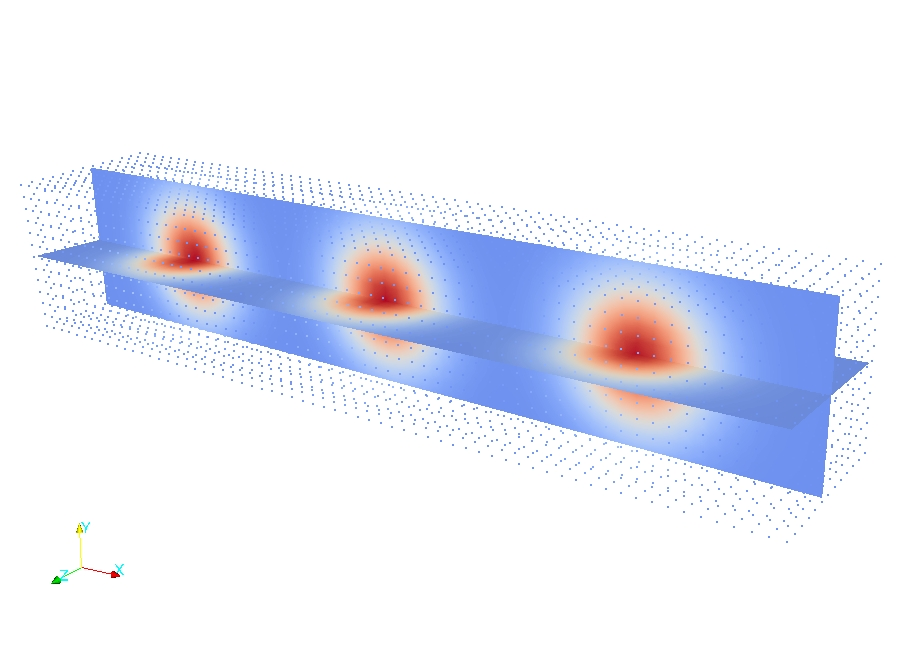}} 
 \\
a) Test iii): horizontal and vertical slices &  b)  Test iv): horizontal and vertical slices \\
{\includegraphics[scale=0.19, trim = 15.0cm 6.0cm 12.0cm 8.0cm, clip=true,]{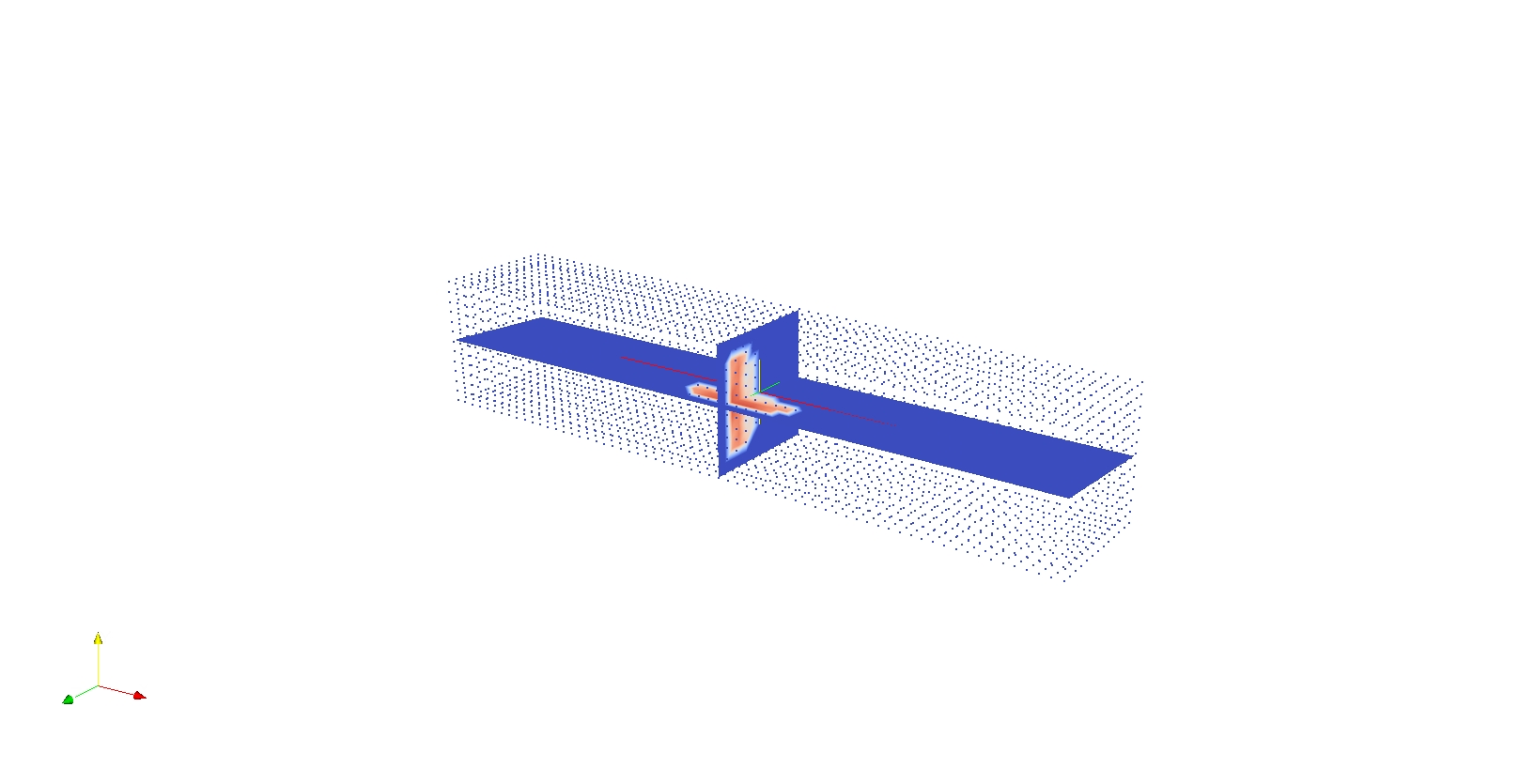}} &
  {\includegraphics[scale=0.19, trim = 15.0cm 6.0cm 12.0cm 8.0cm, clip=true,]{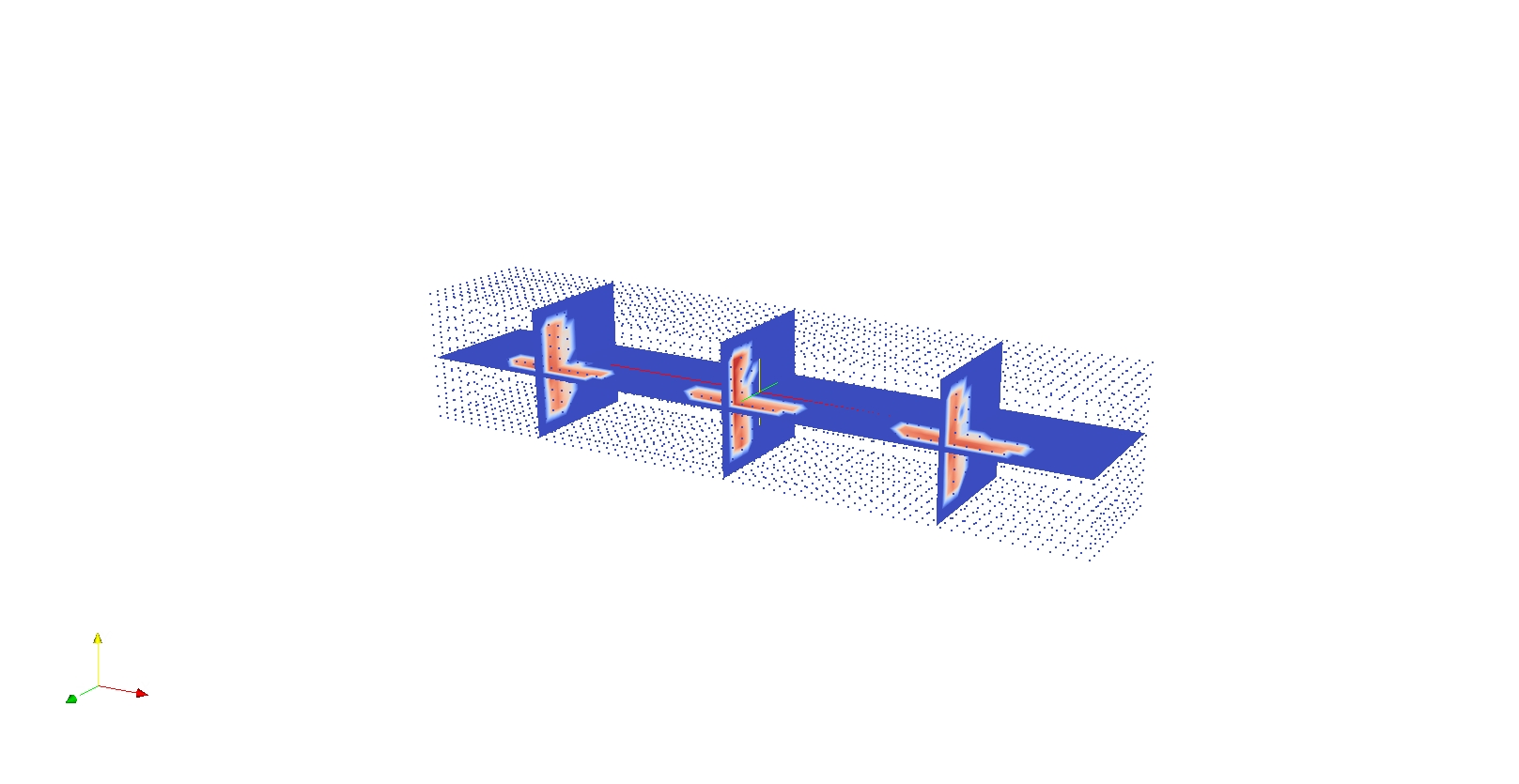}}  \\
c) Test iii): horizontal and vertical slices &  d)  Test iv): horizontal and vertical slices \\
{\includegraphics[scale=0.19, trim = 15.0cm 6.0cm 12.0cm 8.0cm, clip=true,]{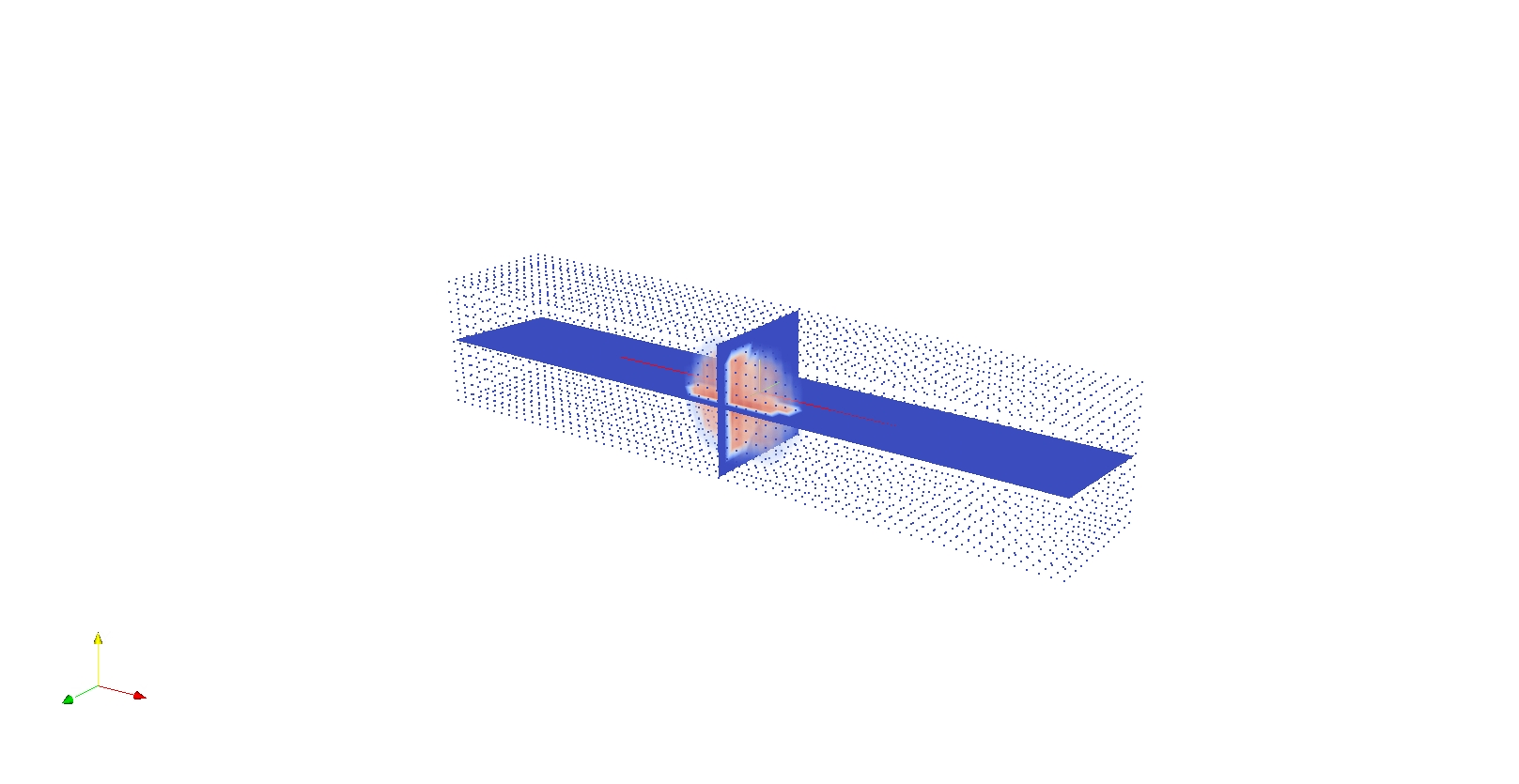}} &
  {\includegraphics[scale=0.19, trim = 15.0cm 6.0cm 12.0cm 8.0cm, clip=true,]{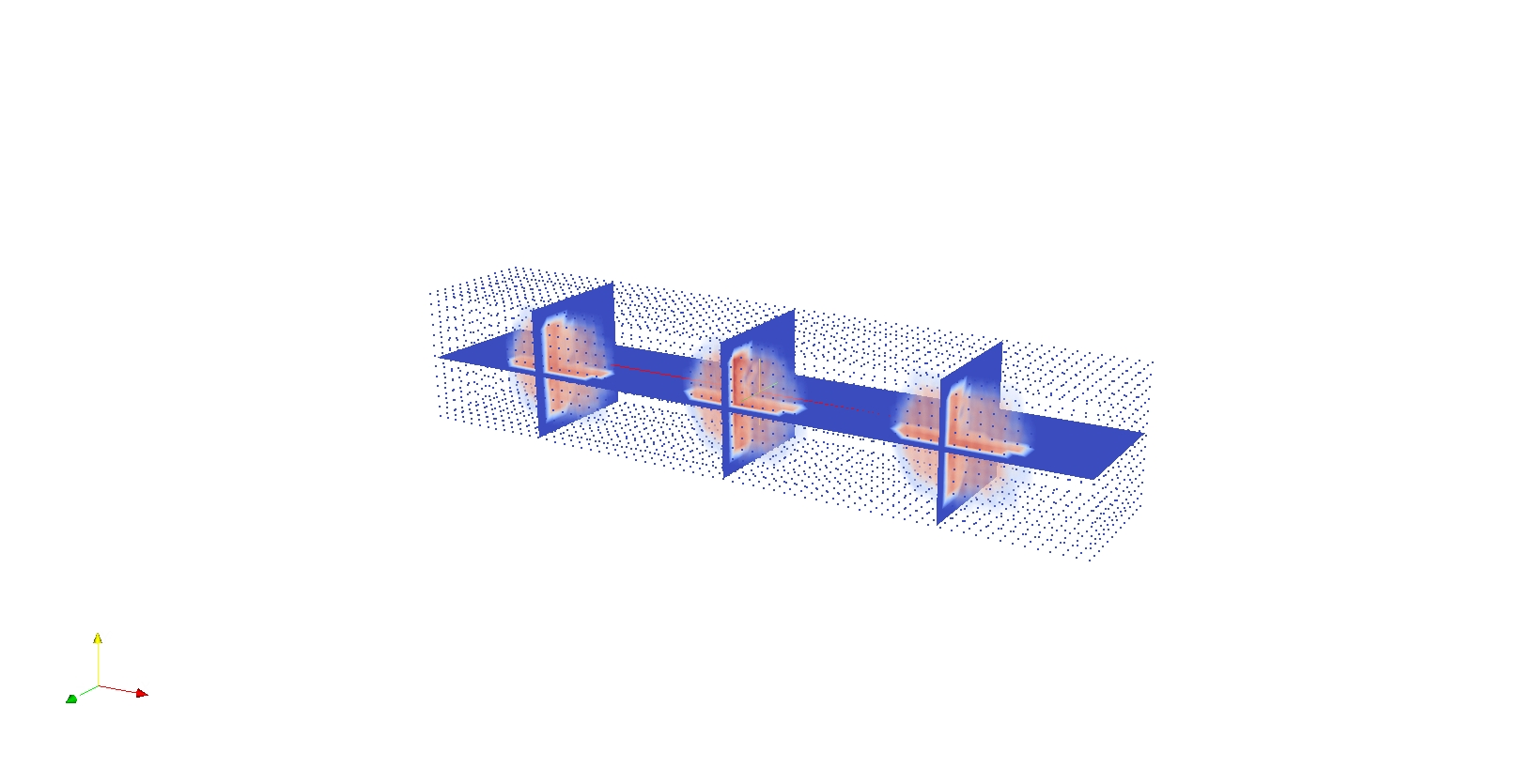}}  \\
e) Test iii): threshold of the solution &  f)  Test iv): threshold of the solution \\
\end{tabular}
 \end{center}
 \caption{ a), b). Slices of the exact  Gaussian functions given by (\ref{1gaussian})  and (\ref{3gaussians}), respectively.  c), d) Slices and e), f) thresholds of the reconstructions. Here, computations were done for the noise $\sigma=10\%$ and $\omega=40$.  }
 \label{fig:exact_gaus}
 \end{figure}

 \subsection{Test  case iii)}

\label{sec:caseiii}

 \begin{figure}
 \begin{center}
 \begin{tabular}{cc}
 {\includegraphics[scale=0.20, trim = 1.0cm 1.0cm 1.0cm 1.0cm,
     clip=true,]{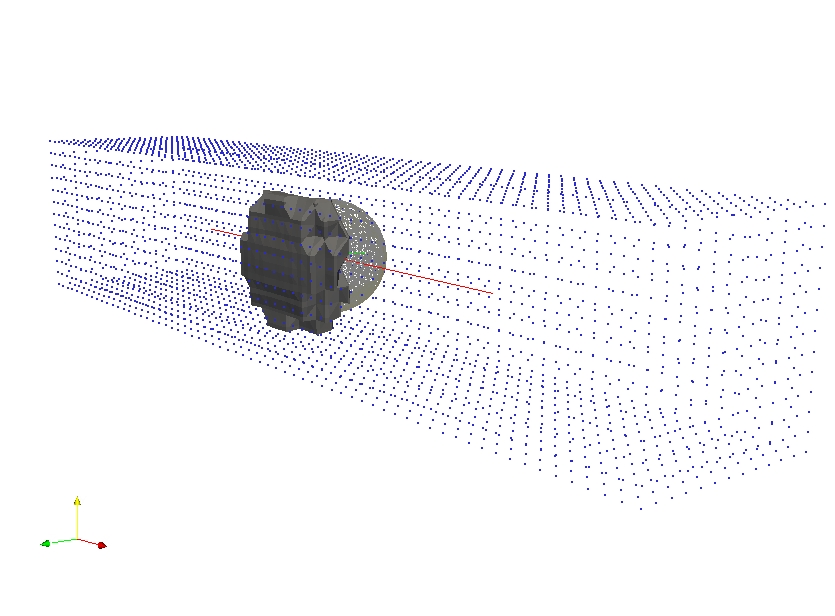}} &
 {\includegraphics[scale=0.20, trim = 1.0cm 1.0cm 1.0cm 1.0cm,
     clip=true,]{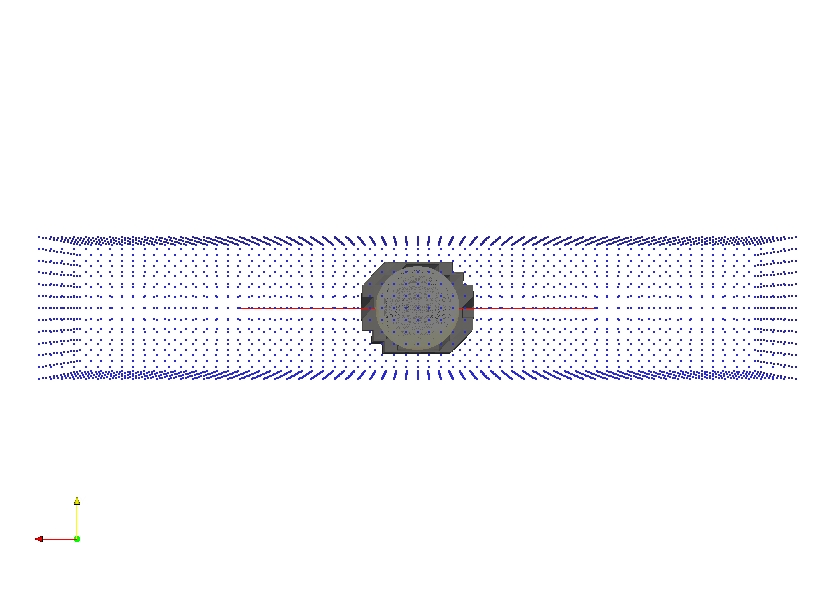}} \\ prospect view & $x_1 x_2$
 view \\ {\includegraphics[scale=0.20, trim = 1.0cm 1.0cm 1.0cm 1.0cm,
     clip=true,]{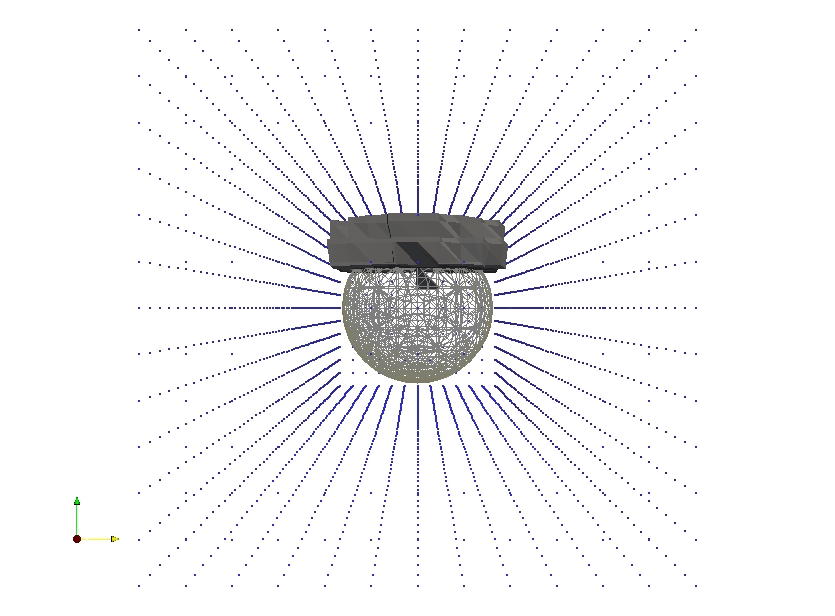}} &
 {\includegraphics[scale=0.20, trim = 1.0cm 1.0cm 1.0cm 1.0cm,
     clip=true,]{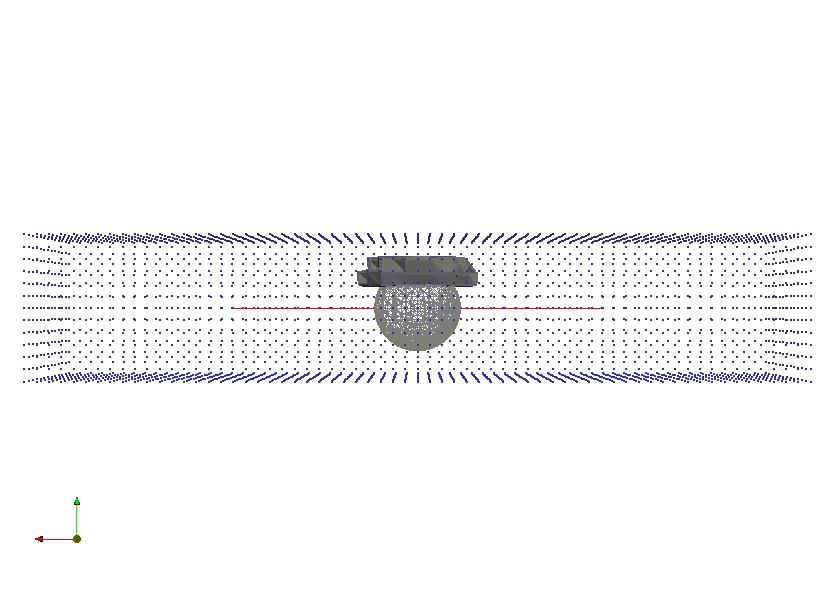}} \\ $x_2 x_3$ view & $x_3 x_1$ view
 \\
\end{tabular}
 \end{center}
 \caption{Test case iii). We present reconstruction of $\tilde{c}$  with $\max_{\Omega_{FEM}} c(x) = 5.91 $ for
   $\omega=40$ in (\ref{f}) with noise level $\sigma=3\%$. The  spherical wireframe of the isosurface
   with exact value of the function (\ref{1gaussian}), which
   corresponds to the value of the reconstructed $\tilde{c}$, is
   outlined by a thin line. }
 \label{fig:1gausnoise3}
 \end{figure}

 \begin{figure}
 \begin{center}
 \begin{tabular}{ccc}
{\includegraphics[scale=0.18,  trim = 1.0cm 1.0cm 1.0cm 1.0cm, clip=true,]{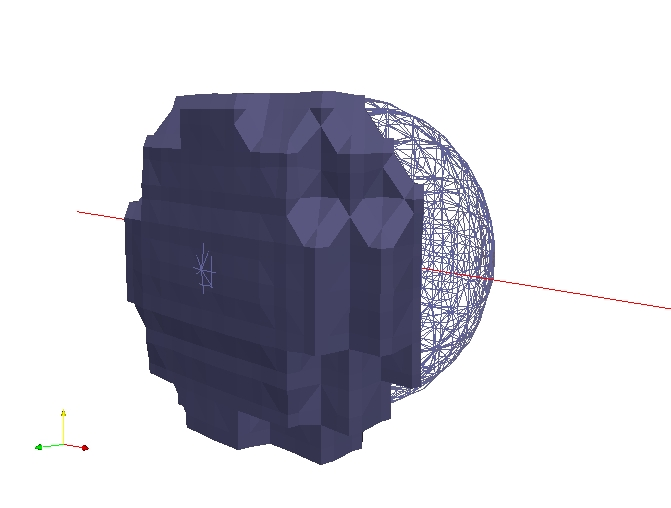}} &
 {\includegraphics[scale=0.18, trim = 1.0cm 1.0cm 1.0cm 1.0cm, clip=true,]{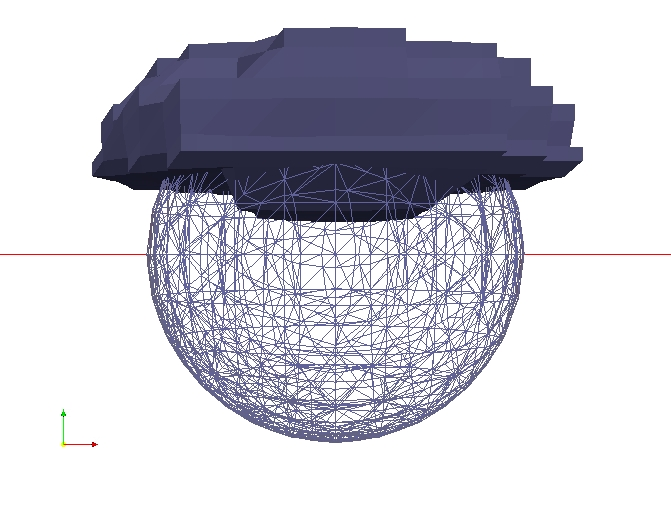}}  &
 {\includegraphics[scale=0.18, trim = 1.0cm 1.0cm 1.0cm 1.0cm, clip=true,]{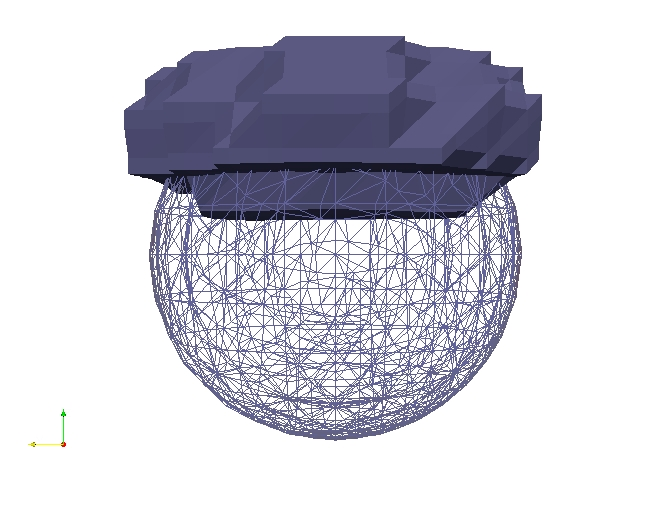}} 
\\

prospect view &  $x_2 x_3 $ view & $x_1 x_3$ view
\end{tabular}
 \end{center}
 \caption{Test case iii). Computed images of reconstructed
   $\max_{\Omega_{FEM}} c(x) = 4.84 $ for $\omega=40$ in (\ref{f}) and
   noise level $\sigma=10\%$.  The spherical
   wireframe of the isosurface with exact value of the function
   (\ref{1gaussian}), corresponding to the value of the
   reconstructed $\tilde{c}$, is outline by a thin line.}
 \label{fig:1gausnoise10}
 \end{figure}

In this numerical test, we reconstruct the conductivity function $c(x)$ which is defined as follows
\begin{equation}\label{1gaussian}
c(x) = 1.0 + 5.0 \cdot \exp^{-({x_1}^2/0.2 + {x_2}^2/0.2 + {x_3}^2/0.2)},
\end{equation}
see Figure \ref{fig:exact_gaus}-a). 
 In this test, we have used noisy
boundary data $u_{\sigma}$ with $\sigma=3\%$ and  $\sigma=10\%$ in (\ref{additive}).
Note that a priori  we have not assumed
 that we know the structure of this function, further we have
assumed that we know the lower bound $c(x) \geq 1$ and that the
reconstructed values of the conductivity belongs to the set of
admissible parameters which is now defined as 
 \begin{equation}\label{admpargaus}
 \begin{split}
  M_{c} \in \{c\in C(\overline{\Omega })|1\leq c(x)\leq 10\}.\\
 \end{split}
 \end{equation}

Figures \ref{fig:exact_gaus}-c), e)  and \ref{fig:1gausnoise10} display
results of the reconstruction of function given by (\ref{1gaussian})
with $\sigma=10\%$ in (\ref{additive}). Quite similar results are
obtained for $\sigma=3\%$ in (\ref{additive}), see figure
\ref{fig:1gausnoise3}. We observe that the location of the maximal
value of the function (\ref{additive}) is imaged very well. It follows
from figure \ref{fig:1gausnoise3} and table 1 that the imaged contrast
in this function is $5.91:1=\max_{\Omega_{FEM}} c_{12}:1 $, where
$n:=\overline{N}=12$ is our final iteration number in the conjugate
gradient method. Similar observation is valid from figure
\ref{fig:1gausnoise10} and table 1 where the imaged contrast is
$4.84:1=\max_{\Omega_{FEM}} c_{16}:1 $, $n:=\overline{N}=16$. However,
from these figures we also observe that because of the data
post-processing procedure (\ref{postproc}) the values of the
background $1.0 + 5.0 \cdot \exp^{-(x^2/0.2 + y^2/0.2 + z^2/0.2)}$ in
(\ref{1gaussian}) are not reconstructed but are smoothed out. Thus, we
are able to reconstruct only maximal values of the function
(\ref{1gaussian}).  Comparison of figures \ref{fig:exact_gaus}-c), e),
\ref{fig:1gausnoise3}, \ref{fig:1gausnoise10} with
figure~\ref{fig:exact_gaus}-a) reveals that it is desirable to improve
shape of the function (\ref{1gaussian}) in $x_3$ direction.  Again,
similar to \cite{BJ, BTKB} we hope that an adaptive finite
element method can refine the obtained images of figure
\ref{fig:1gausnoise10} in order to get better shapes and sizes of the
function (\ref{1gaussian}) in all directions.

 \begin{figure}
 \begin{center}
 \begin{tabular}{cc}
 {\includegraphics[scale=0.20,  trim = 1.0cm 1.0cm 1.0cm 1.0cm, clip=true,]{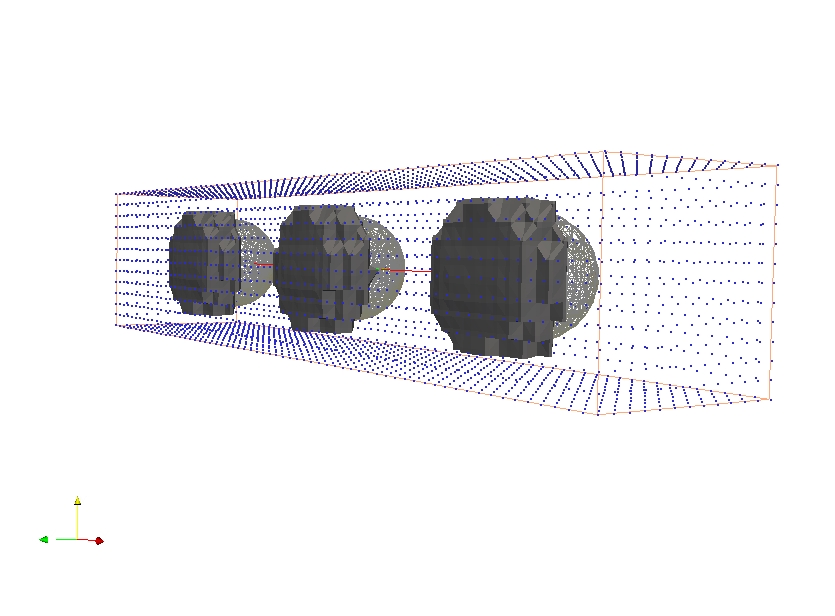}} &
 {\includegraphics[scale=0.20, trim = 1.0cm 1.0cm 1.0cm 1.0cm, clip=true,]{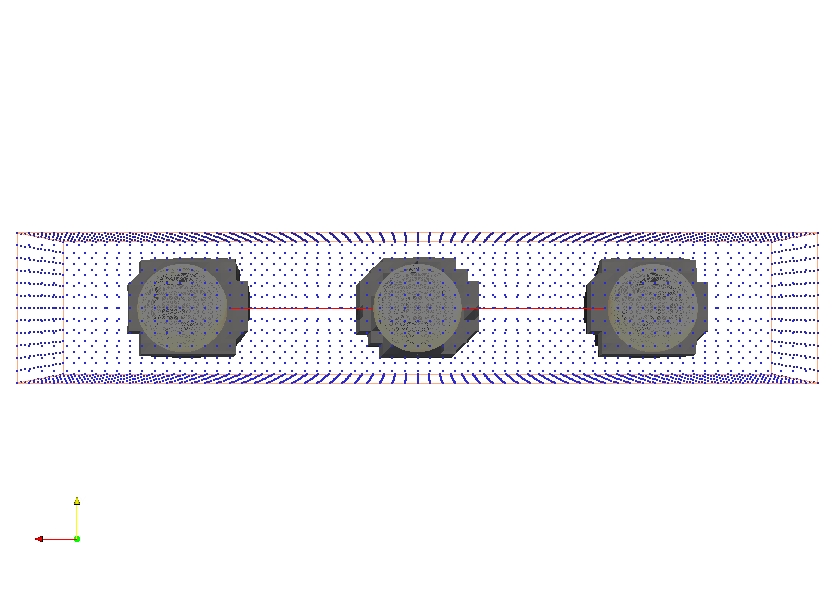}}  
\\
prospect view &  $x_1 x_2$  view \\
 {\includegraphics[scale=0.20,  trim = 1.0cm 1.0cm 1.0cm 1.0cm, clip=true,]{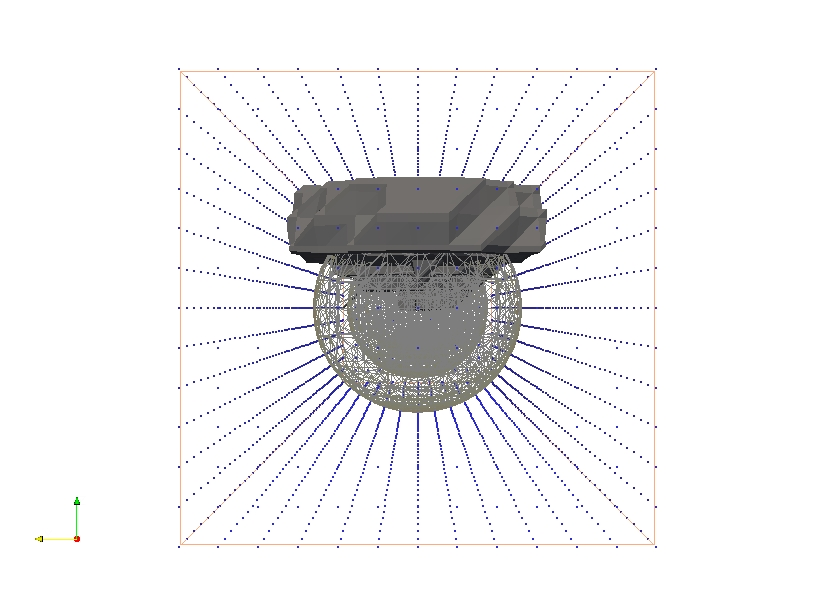}} &
 {\includegraphics[scale=0.20, trim = 1.0cm 1.0cm 1.0cm 1.0cm, clip=true,]{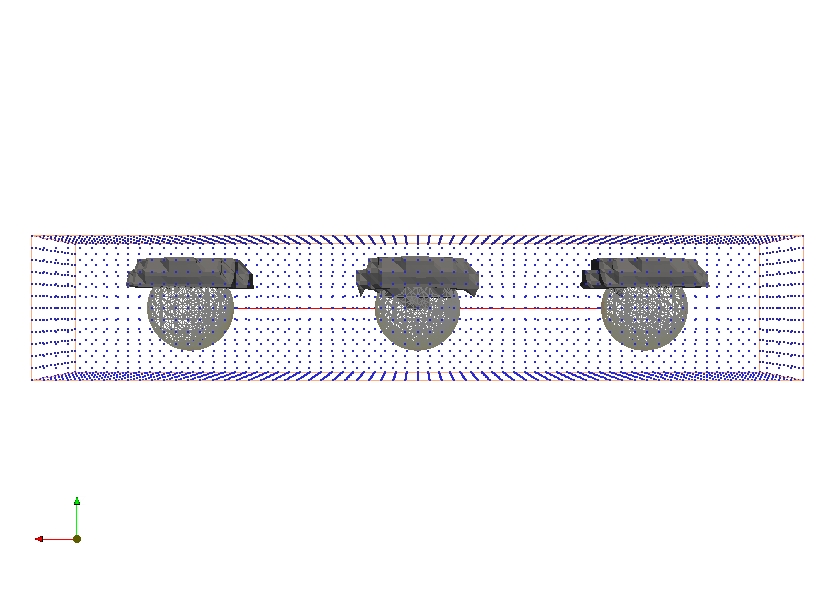}}  
\\
$x_2 x_3$ view  &  $x_3 x_1$  view \\
\end{tabular}
 \end{center}
 \caption{Test case iv).  We present reconstruction of $\tilde{c}$ when
   $\max_{\Omega_{FEM}} c(x) = 5.09 $ for $\omega=40$ in (\ref{f}) and
   noise level $\sigma=3\%$. 
 The spherical
   wireframe of the isosurface with exact value of the function
   (\ref{3gaussians}), which corresponds to the value of the
   reconstructed $\tilde{c}$, is outlined by a thin line.}
 \label{fig:3gausnoise3}
 \end{figure}

 \begin{figure}
 \begin{center}
 \begin{tabular}{cc}
 {\includegraphics[scale=0.20,  trim = 1.0cm 1.0cm 1.0cm 1.0cm, clip=true,]{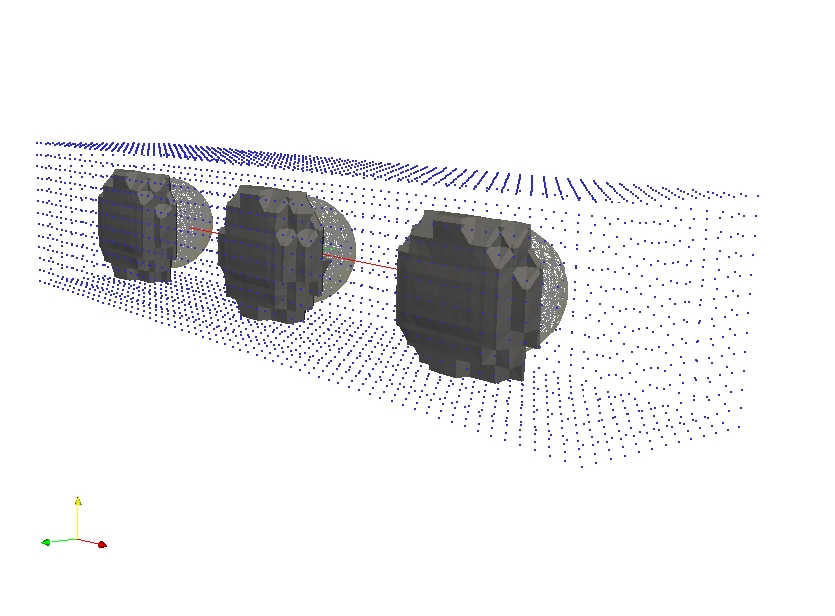}} &
 {\includegraphics[scale=0.20, trim = 1.0cm 1.0cm 1.0cm 1.0cm, clip=true,]{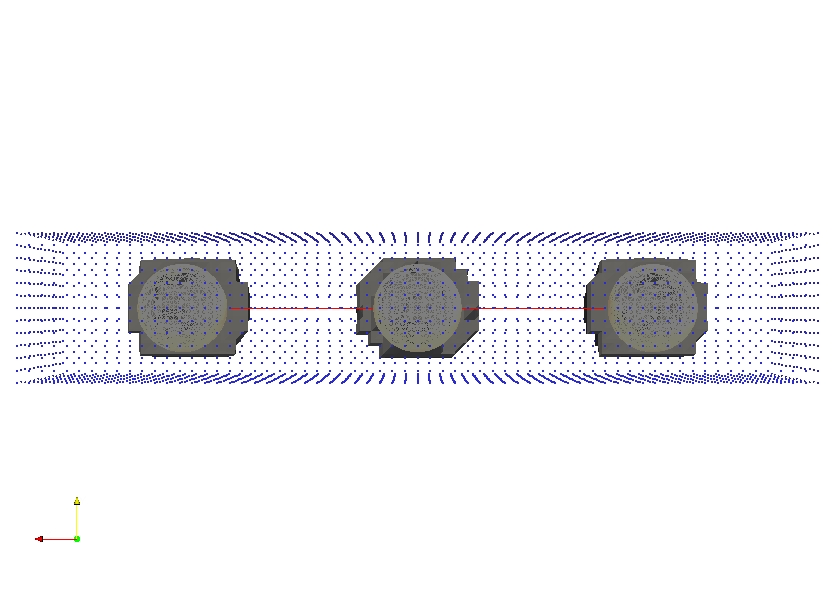}}  
\\
prospect view &  $x_1 x_2$  view \\
 {\includegraphics[scale=0.20,  trim = 1.0cm 1.0cm 1.0cm 1.0cm, clip=true,]{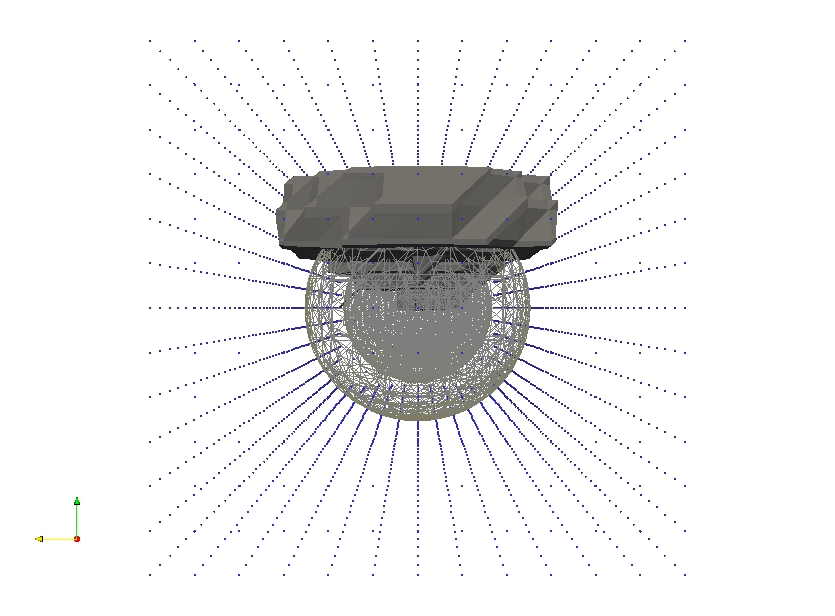}} &
 {\includegraphics[scale=0.20, trim = 1.0cm 1.0cm 1.0cm 1.0cm, clip=true,]{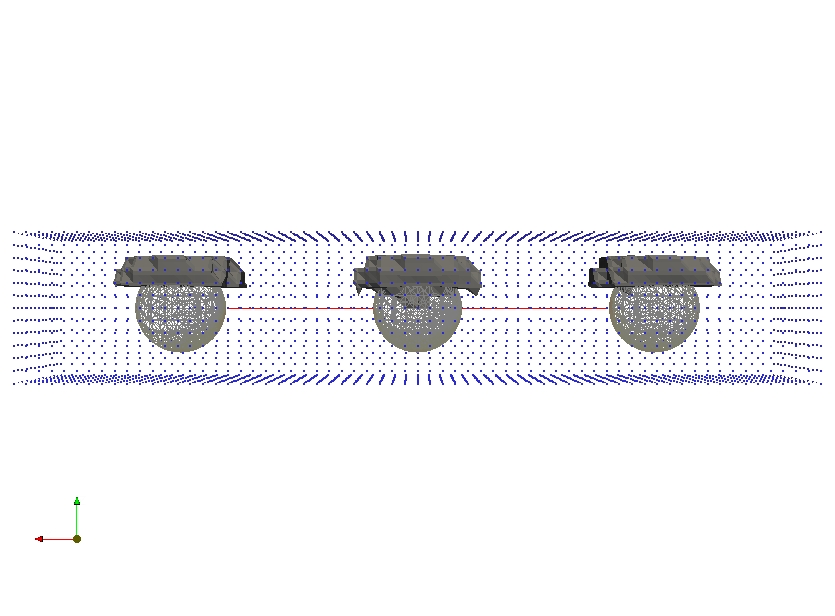}}  
\\
$x_2 x_3$ view  &  $x_3 x_1$  view \\
\end{tabular}
 \end{center}
 \caption{Test case iv).  We present reconstruction of $\tilde{c}$ when
   $\max_{\Omega_{FEM}} c(x) = 5.87$ for $\omega=40$ in (\ref{f}) and    noise level $\sigma=10\%$. The spherical   wireframe of the isosurface with exact value of the function   (\ref{3gaussians}), corresponding to the value of the   reconstructed $\tilde{c}$, is outlined by a thin line.}
 \label{fig:3gausnoise10}
 \end{figure}

  \subsection{Test case iv)}

\label{sec:caseiv}

In our last numerical test we reconstruct the conductivity function
$c(x)$ given by  three sharp Gaussians  such that
\begin{equation}\label{3gaussians}
\begin{split}
c(x) &= 1.0 + 5.0 \cdot \exp^{-((x_1 + 2)^2/0.2 + {x_2}^2/0.2 + {x_3}^2/0.2 )} \\
&+ 5.0 \cdot \exp^{-({x_1}^2/0.2 + {x_2}^2/0.2 + {x_3}^2/0.2 )} + 5.0 \cdot \exp^{-(({x_1}-2)^2/0.2 + {x_2}^2/0.2 + {x_3}^2/0.2 )},
\end{split}
\end{equation}
see Figure \ref{fig:exact_gaus}-b). 
 In this test we again used the noisy
boundary data $u_{\sigma}$ with $\sigma=3\%$ and  $\sigma=10\%$ in (\ref{additive}).
 We assume that the
reconstructed values of the conductivity belongs to the set of
admissible parameters  (\ref{admpargaus}).

Figures \ref{fig:exact_gaus}-d),f), \ref{fig:3gausnoise3} and
\ref{fig:3gausnoise10} show results of the reconstruction of function
given by (\ref{3gaussians}) for $\sigma=3\%$ and $\sigma=10\%$ in
(\ref{additive}), respectively.  We observe that the location of the
maximal value of the function (\ref{3gaussians}) is imaged very
well. It follows from figure \ref{fig:3gausnoise3} and table 1 that
when the noise level is $\sigma=3\%$ then the imaged contrast in this
function is $5.09:1=\max_{\Omega_{FEM}} c_{15}:1 $, where
$n:=\overline{N}=15$ is our final iteration number in the conjugate
gradient method. When the noise level is $\sigma=10\%$ then the imaged
contrast is $5.87:1=\max_{\Omega_{FEM}} c_{18}:1, n:=\overline{N}=18$.

 However, as
in the case iii), the values of the background in (\ref{3gaussians}) are
not reconstructed but are smoothed out and we are able to
reconstruct only maximal values of the three Gaussians in the function (\ref{3gaussians}).
Comparing  figures \ref{fig:exact_gaus}-d),f), \ref{fig:3gausnoise3}, \ref{fig:3gausnoise10}  with
figure~\ref{fig:exact_gaus}-b) we see that it is desirable to improve
shapes of the function (\ref{3gaussians}) in $x_3$ direction.

\section{Discussion and Conclusion}

In this work, we have presented a computational study of the
reconstruction of the conductivity function $c(x)$ in a hyperbolic
problem (\ref{model1}) using Lagrangian approach and a hybrid finite
element/difference method of \cite{BAbsorb}.  As theoretical result, we
have presented estimate of the norms between computed and regularized
solutions of the Tikhonov functional via the $L_2$ norm of the
Fr\'{e}chet derivative of this functional or via the corresponding
Lagrangian. 

In our numerical tests, we have obtained stable  reconstruction
of the conductivity function $c(x)$ in $x_1 x_2$-directions for
frequency $\omega=40$ in the initialization of a plane wave (\ref{f})
and for noise levels $\sigma = 3 \%, 10 \%$ in (\ref{additive}).  However,
size and shape on $x_3$ direction  should 
still be improved in all test cases.  Similar to \cite{BJ, BTKB}
we plan to apply an adaptive finite element method in order to get
better shapes and sizes of the conductivity function $c(x)$ in $x_3$
direction.

Using results of table 1 we can conclude that the computational errors
in the achieved maximal contrast are less in the case of
reconstruction of  smooth functions than in the reconstruction of
small inclusions. This can be explained by involving of discontinuities
in the reconstruction of small inclusions, as well as by having
special geometry in these small inclusions: all of them have different
sizes and locations inside $\Omega_{FEM}$, and thus, achieving  the
exact contrast becomes more difficult task in this case.
The  important observation is that when the
scatters are of different size, especially when the smallest scatterer
is located between larger ones, as in case studies i) and ii),
we note that the smaller scatterer is better reconstructed when it is
located near the observation boundary of the computational domain
$\Omega_{FEM}$.

\section*{Acknowledgement}

The research of L. B.  is partially supported by the sabbatical
program at the Faculty of Science, University of Gothenburg, Sweden,
and the research of K. N.  was supported by the Swedish Foundation for
Strategic Research.  The computations were performed on resources at Chalmers
Centre for Computational Science and Engineering (C3SE) provided by
the Swedish National Infrastructure for Computing (SNIC).

\end{document}